\documentclass[11pt]{article}
\usepackage{amsfonts}
\usepackage[leqno]{amsmath}
\usepackage{graphicx}
\usepackage{epsfig}
\usepackage{caption}
\usepackage{subcaption}
\usepackage{booktabs}
\usepackage{epstopdf}
\usepackage{hyperref}
\hypersetup{
	colorlinks=true,
	linkcolor=blue,
	filecolor=magenta,      
	urlcolor=blue,
	citecolor=red
}
\usepackage{latexsym}
\usepackage{amsmath,amsfonts,amssymb,amsthm,mathrsfs,euscript,makeidx,color}
\usepackage{enumerate}
\allowdisplaybreaks[1]
\usepackage{multirow}

\def\sqr#1#2{{\vcenter{\vbox{\hrule height.#2pt
				\hbox{\vrule width.#2pt height#1pt \kern#1pt \vrule width.#2pt}
				\hrule height.#2pt}}}}

\def\3n{\negthinspace \negthinspace \negthinspace }
\def\2n{\negthinspace \negthinspace }
\def\1n{\negthinspace }
\def\bel{\begin{equation}\label}
	\def\eel{\end{equation}}

\def\dbE{\mathbb{E}}
\def\dbF{\mathbb{F}}

\def\dbN{\mathbb{N}}

\def\dbP{\mathbb{P}}

\def\dbR{\mathbb{R}}
\def\dbS{\mathbb{S}}

\def\sU{\mathscr{U}}

\def\sX{\mathscr{X}}

\def\={\buildrel \triangle \over =}

\def\d{\delta}

\def\l{\lambda}
\def\m{\mu}

\def\si{\sigma}

\def\th{\theta}

\def\i{\infty}
\def\Th{\Theta}

\def\O{\Omega}

\def\cA{{\cal A}}

\def\cH{{\cal H}}

\def\cL{{\cal L}}

\def\cP{{\cal P}}
\def\cQ{{\cal Q}}
\def\cR{{\cal R}}
\def\cS{{\cal S}}

\def\cW{{\cal W}}

\def\BX{{\bf X}}

\def\Bu{{\bf u}}

\def\Th{\Theta}

\def\O{\Omega}

\def\ms{\medskip}

\def\hb{\hbox}

\def\limsup{\mathop{\overline{\rm lim}}}

\def\lan{{\langle}}
\def\ran{{\rangle}}

\def\h{\widehat}
\def\wt{\widetilde}

\def\cd{\cdot}

\def\les{\leqslant}
\def\ges{\geqslant}

\def\({\Big (}
\def\){\Big )}
\def\[{\Big[}
\def\]{\Big]}
\def\lan{\langle}
\def\ran{\rangle}

\def\bde{\begin{definition}\label}
	\def\ede{\end{definition}}

	\def\bel{\begin{equation}\label}
		\def\ee{\end{equation}}
	\def\bt{\begin{theorem}\label}
		\def\et{\end{theorem}}
	\def\bc{\begin{corollary}\label}
		\def\ec{\end{corollary}}
	\def\bl{\begin{lemma}\label}
		\def\el{\end{lemma}}
	\def\bp{\begin{proposition}\label}
		\def\ep{\end{proposition}}
	\def\bex{\begin{example}\label}
		\def\ex{\end{example}}
	\def\bas{\begin{assumption}}
		\def\eas{\end{assumption}}
	\def\br{\begin{remark}\label}
		\def\er{\end{remark}}
	\def\ba{\begin{array}}
		\def\ea{\end{array}}
	\def\ed{\end{document}}

\def\square#1{\vbox{\hrule\hbox{\vrule height#1%
			\kern#1\vrule}\hrule}}
\def\rectangle#1#2{\vbox{\hrule\hbox{\vrule height#1%
			\kern#2\vrule}\hrule}}

\font\tenbb=msbm10 \font\sevenbb=msbm7 \font\fivebb=msbm5

\newfam\bbfam
\scriptscriptfont\bbfam=\fivebb \textfont\bbfam=\tenbb
\scriptfont\bbfam=\sevenbb

\newtheorem{theorem}{Theorem}[section]
\newtheorem{corollary}[theorem]{Corollary}

\newtheorem{lemma}[theorem]{Lemma}
\newtheorem{proposition}[theorem]{Proposition}
\newtheorem{assumption}[theorem]{Assumption}

\theoremstyle{definition}
\newtheorem{definition}[theorem]{Definition}

\newtheorem{remark}[theorem]{Remark}

\newtheorem{example}{Example}[section]

\makeatletter

\@addtoreset{equation}{section}
\makeatother

\usepackage{xcolor}
\makeatletter

\input pdf-trans
\newbox\qbox
\def\usecolor#1{\csname\string\color@#1\endcsname\space}
\newcommand\bordercolor[1]{\colsplit{1}{#1}}
\newcommand\fillcolor[1]{\colsplit{0}{#1}}
\newcommand\outline[1]{\leavevmode%
	\def\maltext{#1}%
	\setbox\qbox=\hbox{\maltext}%
	\boxgs{Q q 2 Tr \thickness\space w \fillcol\space \bordercol\space}{}%
	\copy\qbox%
}
\makeatother
\newcommand\colsplit[2]{\colorlet{tmpcolor}{#2}\edef\tmp{\usecolor{tmpcolor}}%
	\def\tmpB{}\expandafter\colsplithelp\tmp\relax%
	\ifnum0=#1\relax\edef\fillcol{\tmpB}\else\edef\bordercol{\tmpC}\fi}
\def\colsplithelp#1#2 #3\relax{%
	\edef\tmpB{\tmpB#1#2 }%
	\ifnum `#1>`9\relax\def\tmpC{#3}\else\colsplithelp#3\relax\fi
}
\bordercolor{black}
\fillcolor{white}
\def\thickness{.3}

\def\hTh{\outline{$\Theta$}}

\title{Ergodicity and turnpike properties of linear-quadratic mean field control problems}

\author{Erhan Bayraktar
\thanks{Department of Mathematics, University of Michigan, Ann Arbor, MI 48109, USA. {\tt erhan@umich.edu.} This
author is funded in part by National Science Foundation under grant DMS-2507940 and in part by the Susan M.
Smith Professorship.}
~~~~Jiamin Jian
\thanks{Department of Mathematics, University of Michigan, Ann Arbor, MI 48109, USA. {\tt jiaminj@umich.edu.}
}
}

\date{}

\begin{document}

\maketitle


\begin{abstract}
We study the asymptotic behavior of linear-quadratic mean field stochastic optimal control problems, in which the expectation of the control enters both the drift and diffusion coefficients, and the cost functional incorporates general quadratic terms involving the expectations of both the state and the control. By formulating an ergodic control framework, we characterize the convergence between the finite-horizon control problem and its ergodic counterpart.
Leveraging these convergence results, we establish the turnpike property for the optimal pairs, demonstrating that solutions to the finite-horizon control problem remain exponentially close to the ergodic optimal pair except near the temporal boundaries. 
This result reveals the intrinsic connection between long-term dynamics and their asymptotic behavior in mean field control systems.
\end{abstract}


\section{Introduction}

\subsection{Finite-horizon mean field control problem}

Let $\{W(t)\}_{t \ges 0}$ be a one-dimensional standard Brownian motion defined on a complete filtered probability space $(\Omega, \mathcal F, \dbF,\dbP)$ satisfying the usual conditions, with $\dbF=\{\mathcal F_t\}_{t\ges0}$ being the natural filtration of $W(\cd)$ augmented by all the $\dbP$-null sets in $\mathcal F$. Let $\dbR^n$ represent the standard $n$-dimensional real Euclidean space. For any given $n, m \in \dbN$, we denote by $\dbR^{n \times m}$ the Euclidean space of all $n \times m$ real matrices. Consider the following $n$-dimensional controlled linear mean field stochastic differential equation (MF-SDE):
\begin{equation}
\label{eq:state}
\begin{cases}
\vspace{4pt}
\displaystyle
d X(t) = \big\{AX(t) + \bar A \dbE[X(t)] + Bu(t) + \bar{B} \dbE[u(t)] +  b \big\} dt \\
\vspace{4pt}
\displaystyle
\hspace{0.8in} + \big\{CX(t) + \bar C \dbE[X(t)] + Du(t) + \bar{D} \dbE[u(t)] + \sigma \big\} d W(t), \quad t \ges 0, \\
X(0) = x,
\end{cases}
\end{equation}
where $A, \bar A, C, \bar C \in \dbR^{n \times n}$ and $B, \bar{B}, D, \bar{D} \in \dbR^{n \times m}$ are called the coefficients of the system, and $b, \sigma \in \dbR^n$ are called the non-homogeneous terms. In the above, $X(\cd)$ is the state process, which takes values in $\dbR^n$, and $u(\cdot)$ is the control process, valued in $\dbR^m$. We refer to MF-SDE \eqref{eq:state} as the state equation.

To proceed, we introduce the following spaces for the control processes. Let $T>0$ denote a large time horizon and
\begin{equation*}
\begin{aligned}
&\sU[0,T] = \Big\{u:[0,T] \times \O \to \dbR^m \Bigm| u(\cd) \hb{ is $\dbF$-progressively measurable, and} \\
&\hspace{2.5in} \dbE \Big[\int_0^T|u(t)|^2dt \Big]<\i\Big\}, \\
& \sU_{loc}[0, \infty) = \bigcap_{T > 0} \sU[0, T], \\
& \sU[0, \infty) := L^2_{\dbF}(0, \infty; \mathbb R^m) = \Big\{u \in \sU_{loc}[0, \infty) \Bigm| \dbE \Big[\int_0^\infty |u(t)|^2dt \Big]<\i\Big\}.
\end{aligned}
\end{equation*}
In a similar manner, we define the following spaces for the state processes to facilitate the analysis:
\begin{equation*}
\begin{aligned}
&\sX[0,T] = \Big\{X:[0,T] \times \O \to \dbR^n \Bigm| X(\cd) \hb{ is $\dbF$-adapted, $t \mapsto X(t, \omega)$ is continuous, and} \\
&\hspace{2.5in} \dbE\Big[\sup_{t \in [0, T]} |X(t)|^2 \Big ] < \i \Big\}, \\
& \sX_{loc}[0, \infty) = \bigcap_{T > 0} \sX[0, T], \\
& \sX[0, \infty) = \Big\{X \in \sX_{loc}[0, \infty) \Bigm| \dbE \Big[\int_0^\infty |X(t)|^2dt \Big]<\i\Big\}.
\end{aligned}
\end{equation*}
It is clear that under some mild conditions, for any initial state $x \in \dbR^n$ and any admissible control $u(\cd) \in \sU[0, T]$, the state equation \eqref{eq:state} admits a unique solution
$X(\cd) = X(\cd \ ; x, u(\cd)) \in \sX[0, T]$.

In what follows, we denote $\dbS^n$ ($\dbS^n_{+}$, $\dbS^n_{++}$) be the collection of $n \times n$
real-valued symmetric (positive semi-definite, positive definite) matrices. Moreover, for matrices $M, N \in \dbS^n$, we write $M \ges N$ (respectively, $M > N$) if $M - N$ is positive semi-definite (respectively, positive definite).
To evaluate the performance of the control, we introduce the following mean field quadratic cost functional:
\begin{equation}
\label{eq:cost_finite}
J_{T}(x; u(\cd)) = \dbE \Big[\int_0^T f(X(t), \dbE[X(t)], u(t), \dbE[u(t)]) dt \Big],
\end{equation}
where
\begin{equation}
\label{eq:running_cost}
\begin{aligned}
f(x, \bar x, u, \bar{u}) &= \left\lan \begin{pmatrix}
Q & S^\top \\ S & R
\end{pmatrix} \begin{pmatrix}
x \\ u
\end{pmatrix}, \begin{pmatrix}
x \\ u
\end{pmatrix} \right\ran + 2 \left\lan \begin{pmatrix}
q \\ r
\end{pmatrix}, \begin{pmatrix}
x \\ u
\end{pmatrix}\right\ran \\
& \hspace{0.3in} + \left\lan \begin{pmatrix}
\bar{Q} & \bar{S}^\top \\ \bar{S} & \bar{R}
\end{pmatrix} \begin{pmatrix}
\bar{x} \\ \bar{u}
\end{pmatrix}, \begin{pmatrix}
\bar{x} \\ \bar{u}
\end{pmatrix} \right\ran + 2 \left\lan \begin{pmatrix}
\bar{q} \\ \bar{r}
\end{pmatrix}, \begin{pmatrix}
\bar{x} \\ \bar{u}
\end{pmatrix}\right\ran,
\end{aligned}
\end{equation}
with $Q, \bar Q \in \dbS^{n}$, $S, \bar{S} \in \dbR^{m \times n}$, $R, \bar{R} \in \dbS^{m}$ are constant matrices, and $q, \bar{q} \in \dbR^n, r, \bar{r} \in \dbR^m$ are constant vectors. In the above, $\lan \cd, \cd \ran$ denotes the inner product of two vectors, while the superscript $\top$ represents the transpose of matrices. Moreover, we denote by $|\cdot|$ the Euclidean norm in the corresponding Euclidean space, and by $\|A\| = \sqrt{\text{trace}(A A^\top)}$ the Frobenius norm of a matrix $A$, where $\text{trace}(\cd)$ is the trace operator. The finite-horizon stochastic linear-quadratic (LQ) mean field optimal control problem is formulated as follows.

\ms

{\bf Problem (MFLQ)$^T$.} For a given initial state $x\in \dbR^n$, find an open-loop optimal control $u_T(\cd) \in\sU[0,T]$ such that
\begin{equation}
\label{eq:MFLQ_finite}
V_T(x) := J_T(x; u_T(\cd)) = \inf_{u(\cd) \in \sU[0, T]} J_T(x; u(\cd)).
\end{equation}

If there exists a control $u_T(\cd) \in\sU[0,T]$ satisfying \eqref{eq:MFLQ_finite}, we say that Problem (MFLQ)$^T$ is open-loop solvable, and $u_T(\cd)$ is referred to as an open-loop optimal control. The corresponding open-loop optimal state process is denoted by $X_T(\cd):= X(\cd\,; x, u_T(\cd))$. The function $V_T(\cd)$ is called the value function of Problem (MFLQ)$^T$, and $(X_T(\cd), u_T(\cd))$ is referred to as an open-loop optimal pair.

\ms

The homogeneous version of Problem (MFLQ)$^T$ (i.e., the case of $b = \sigma = q = \bar{q} = 0$ and $r = \bar{r} = 0$) was initially studied by \cite{Yong13} under the standard condition in LQ theory, in which the coefficients of the state equation and the weighting matrices of the cost functional are allowed to be time-dependent. The results of \cite{Yong13} were subsequently extended by \cite{Sun17} and \cite{LSY16} to uniformly convex cost functionals with non-homogeneous terms, and by \cite{HLY15} to the infinite-horizon case. Building on this foundation, numerous researchers have explored the mean field LQ optimal control problem from various perspectives. For instance, \cite{BP19} employed a weak martingale approach to study the mean field LQ optimal control problem; \cite{LLY20} introduced the concept of relaxed compensator to address indefinite problems; and \cite{Qi20} investigated a mean field LQ optimal control problem for stochastic evolution equations. Furthermore, \cite{Pham-Wei-2018} studied the stochastic optimal control problem for mean field SDEs, established the dynamic programming principle in its general form, and presented applications to the linear-quadratic mean field control problem.

It is well-established that, under suitable conditions, Problem (MFLQ)$^T$ (even when the coefficients $b(\cd), \sigma(\cd), q(\cd), r(\cd)$ are allowed to be random) admits a unique open-loop optimal control $u_T(\cd) \in \sU[0, T]$. Moreover, this optimal control admits a closed-loop representation through the solution of a system of two Riccati differential equations. For further details, we refer the reader to \cite{Sun17} and \cite{Sun-Yong-2020}.

\subsection{Ergodic control problem}
\label{s:ergodic_control_problem_intro}

A natural question arises: how does the problem evolve when the time horizon is extended to the interval $[0, \infty)$? In particular, we now consider the same state equation \eqref{eq:state} with the following cost functional:
\begin{equation}
\label{eq:cost_infinity}
J_{\infty}(x; u(\cd)) = \dbE \Big[\int_0^{\infty} f(X(t), \dbE[X(t)], u(t), \dbE[u(t)]) dt \Big].
\end{equation}
Such an infinite-horizon mean field control problem has been studied in \cite{HLY15} in the homogeneous setting (i.e., the coefficients $b, \sigma, q, \bar{q}, r$ and $\bar{r}$ are all zero, with the additional restriction $S = \bar{S} = 0$) under suitable stabilizability assumption for the corresponding homogeneous system. It was further investigated in \cite{Bayraktar-Zhang-2023} through the solvability of the associated infinite-horizon McKean--Vlasov forward-backward SDEs. More recently, the long-term average reward control problem for McKean--Vlasov dynamics was analyzed in \cite{Fuhrman-Ruda-2025}, where both the ergodic value $\lambda$ and an associated function $\phi$ were constructed, with $\phi$ serving as a viscosity solution to an ergodic Hamilton--Jacobi--Bellman equation of elliptic type on the Wasserstein space.

Now, in the non-homogeneous setting, particularly when $b \neq 0$ or $\sigma \neq 0$, the global integrability condition of the state process and running cost functional is not satisfied, i.e., $X(t; x, u(\cd)) \notin \sX[0, \infty)$. Consequently, even if the corresponding homogeneous system is stabilizable, the state process $X(t; x, u(\cd))$ and thus the running cost functional will not approach zero as $t \to \infty$. 
As a result, the cost functional in \eqref{eq:cost_infinity} will generally not be well-defined. Therefore, the corresponding LQ mean field control problem is not properly formulated in the traditional sense.

However, under some stabilizability condition on the homogeneous state process, as in \cite{Fuhrman-Ruda-2025}, we can introduce the following so-called Ces\`aro mean type cost functional:
\begin{equation}
\label{eq:Cesaro_cost_function}
\bar J_T(x; u(\cd)) = \frac{1}{T} J_T(x; u(\cd)) = \frac{1}{T} \dbE \Big[\int_0^T f(X(t), \dbE[X(t)], u(t), \dbE[u(t)]) dt \Big],
\end{equation}
where $J_{T}(x; u(\cd))$ is given by \eqref{eq:cost_finite}. Then, we define the ergodic cost as
\begin{equation}
\label{eq:ergodic_cost}
\bar c := \lim_{T \to\infty} \frac{1}{T} V_{T}(x) = \lim_{T \to \infty} \bar J_T(x; u_T(\cd))
\end{equation}
provided that the limit exists, where $V_T(x)$ is the value function of Problem (MFLQ)$^T$. Moreover, we now consider the following ergodic control problem:

\ms

{\bf Problem (EC).} For a given initial state $x\in \dbR^n$, determine a pair $(U(\cd), c)$ satisfying
\begin{equation}
\label{eq:ergodic_control_problem}
\begin{aligned}
& c := \inf_{u(\cd) \in \sU} \limsup_{T \to \infty} \frac{1}{T} \int_0^T \dbE[f(X(t), \dbE[X(t)], u(t), \dbE[u(t)])] dt, \\
& U(x) := \inf_{u(\cd) \in \sU} \limsup_{T \to \infty} \int_0^T \dbE[f(X(t), \dbE[X(t)], u(t), \dbE[u(t)]) - c] dt,
\end{aligned}
\end{equation}
where $\sU$ is a proper class of admissible control to be defined later.

\ms

If there exists $\bar u(\cd) \in \sU$ attaining the two infimums above, then it is called an optimal control, and the associated state process $\bar X(\cd) := \bar X(\cd \,; x, \bar u(\cd))$ is called an optimal state. The function $U(\cd)$ is referred to as the value function of Problem (EC), while the pair $(\bar X(\cd), \bar u(\cd))$ is termed an optimal pair. A solution to Problem (EC) is defined as a 4-tuple $\{U(\cd), c, \bar X(\cd), \bar u(\cd) \}$.

\ms

The above ergodic control problem \eqref{eq:ergodic_control_problem} is referred to as the probabilistic cell problem in \cite{Jian-Jin-Song-Yong-2024} for the standard linear-quadratic control problem, as it provides a probabilistic interpretation of the so-called cell problem related to the homogenization of the Hamilton--Jacobi equation (see \cite{Tran2021}). Moreover, the ergodic control problem with an Abel mean-type cost functional over the infinite horizon $[0, \infty)$ in the LQ setting has been studied in \cite{Mei-Wei-Yong-2021}, under the assumption that the homogeneous state equation is stabilizable.

In this paper, we establish the solvability of Problem (EC) and investigate the turnpike properties that connect the optimal pair $(X_T(\cd), u_T(\cd))$ of Problem (MFLQ)$^T$ to the optimal pair $(\bar X(\cd), \bar u(\cd))$ of Problem (EC). 
In the following subsections, we first give a concise overview of turnpike properties and then summarize our main contributions. 

\subsection{Turnpike properties}

A notable characteristic of (deterministic) optimal control problems over long time horizons is the tendency of the optimal pair to closely approximate that of a related static optimization problem under certain conditions. This phenomenon, known as the turnpike property, derives its name from the U.S. highway system and is closely tied to the ergodic control problem. The turnpike property was first observed by Ramsey in 1928 while studying economic growth problems over an infinite horizon \cite{Ramsey1928}. It was further developed in 1945 by von Neumann \cite{Neumann1945}, and the terminology ``turnpike property" was introduced in 1958 by Dorfman, Samuelson, and Solow \cite{Dorfman-Samuelson-Solow-1958}. Since then, the turnpike phenomenon has attracted considerable attention in both finite and infinite-dimensional problems, in the context of deterministic discrete-time and continuous-time systems. For more on this, see the relevant papers \cite{McKenzie1976, Zaslavski2011, Porretta-Zuazua-2013, Trelat-Zuazua-2015, Lou-Wang-2019, Breiten-Pfeiffer-2020, Esteve-Kouhkouh-Pighin-Zuazua-2022}, as well as the surveys in \cite{Carlson-Haurie-Leizariwitz-2012, Zaslavski2005, Zaslavski2019}. 

In the deterministic case (when $\bar A = C = \bar C = 0$, $\bar{B} = D = \bar{D} = 0$, $\sigma = 0$), the turnpike property is rigorously established in \cite{Porretta-Zuazua-2013}. Specifically, it is shown that there exist some absolute constants $\l$ and $K>0$ such that
\begin{equation*}
|X_T(t)-x^*| + |u_T(t)-u^*|\les K \big(e^{-\l t}+e^{-\l(T-t)} \big),\quad  \forall t \in [0,T],
\end{equation*}
where $(x^*,u^*) \in \dbR^n \times \dbR^m$ is a pair solving the following static optimization problem
\begin{equation*}
\begin{cases}
\vspace{4pt}
\displaystyle \hb{minimize } L(x,u) := \lan Qx, x\ran + 2\lan Sx,u\ran + \lan Ru,u\ran + 2\lan q,x\ran + 2\lan r,u\ran, \\
\hb{subject to } Ax+Bu+b=0.
\end{cases}
\end{equation*}
This result further implies that for all $\d \in (0, 1/2)$ and $t \in [\d T, (1-\d)T]$,
\begin{equation*}
|X_T(t) - x^*| + |u_T(t) - u^*| \les 2 K e^{- \l \d T}.
\end{equation*}
Thus, when $T$ is large, the open-loop optimal pair $(X_T(\cd), u_T(\cd))$ can be well approximated by $(x^*,u^*)$ over the majority of the time horizon $[0, T]$. 

A recent paper by Sun, Wang, and Yong (see \cite{Sun-Wang-Yong-2022}) was the first to address the turnpike property for stochastic LQ optimal control problems. Subsequently, Sun and Yong extended these results in \cite{Sun-Yong-2024}, establishing the turnpike property for mean field linear stochastic systems under the assumption that the homogeneous state equation is stabilizable.
Furthermore, they also explored turnpike properties for stochastic LQ optimal control problems with periodic coefficients in \cite{Sun-Yong-2024-periodic}. Specifically, Theorem 3.2 of \cite{Sun-Yong-2024} (with $\bar{q} = 0$ and $\bar{r} = 0$) shows that there exist some constants $K, \l > 0$ such that
\begin{equation}
\label{eq:turnpike_stochastic}
\dbE\big[|X_T(t)-\BX^*(t)|^2+|u_T(t)-\Bu^*(t)|^2\big]\les K \big(e^{-\l t}+e^{-\l(T-t)} \big), \quad \forall t \in [0,T],
\end{equation}
where $\BX^*(\cd)$ and $\Bu^*(\cd)$ are two constructed stochastic processes independent of $T$, satisfying
$$\dbE[\BX^*(t)]=x^*, \quad \dbE[\Bu^*(t)]=u^*,$$
with $(x^*,u^*)$ being the solution of the following static optimization problem:
\begin{equation}
\label{eq:static_optimization_problem}
\begin{cases}
\vspace{4pt}
\displaystyle \hb{minimize } L(x,u) := \lan (Q + \bar Q) x, x\ran + 2\lan (S + \bar{S}) x, u \ran + \lan (R + \bar{R}) u, u \ran + 2\lan q,x\ran + 2\lan r,u\ran \\
\vspace{4pt}
\displaystyle
\hspace{1.6in} + \lan P((C + \bar C) x + (D + \bar{D}) u + \si), (C + \bar C) x + (D + \bar{D}) u + \si \ran, \\
\displaystyle
\hb{subject to } (A + \bar A) x + (B + \bar{B}) u + b = 0.
\end{cases}
\end{equation}
Here, $P$ is the unique positive definite solution to the algebraic Riccati equation
\begin{equation}
\label{eq:ARE_P}
\begin{aligned}
& PA+A^\top P+C^\top PC+Q\\
& \hspace{0.5in} -(PB+C^\top PD+S^\top)(R+D^\top PD)^{-1}(PB+C^\top PD+S^\top)^\top=0.
\end{aligned}
\end{equation}
The property \eqref{eq:turnpike_stochastic} is referred to as stochastic turnpike property. In addition, by introducing the so-called probability cell problem, the recent paper \cite{Jian-Jin-Song-Yong-2024} establishes a connection between the cell problem and the ergodic cost problem, and reveals the turnpike properties of the LQ stochastic optimal control problems from multiple perspectives.

\subsection{Main contributions and organization of this paper}

The main contribution of this work is to establish turnpike properties for Problem (MFLQ)$^T$ through its associated ergodic control problem, Problem (EC); see Theorem \ref{t:turnpike_property}. Moreover, we identify a fundamental connection between the Bellman equation \eqref{eq:cell_problem} and the ergodic cost \eqref{eq:ergodic_cost} in Theorem \ref{t:ergodic_control_problem}, showing that the constant arising in the solution to the Bellman equation coincides precisely with the ergodic cost. To the best of our knowledge, this is the first result deriving turnpike properties for finite-horizon LQ mean field control problems directly through their associated ergodic control problem.

The proof relies on a detailed solvability and convergence analysis. We first establish the solvability of Problem (MFLQ)$^T$ and Problem (EC) via an analytical approach, and derive explicit representations of the corresponding optimal pairs and value functions; see Lemma \ref{l:finite_time_control} and Theorem \ref{t:ergodic_control_problem}. We then characterize the convergence between these two problems. Specifically, we prove the convergence of coefficients appearing in the value functions and optimal controls in Lemma \ref{l:estimates}, which plays a key role in establishing the turnpike behavior.

Our approach differs substantially from the existing literature. Compared with \cite{Jian-Jin-Song-Yong-2024}, we extend the analysis from the classical LQ control framework to the more challenging mean field setting, where new arguments are required to address the system dynamics, the solvability of the Hamilton--Jacobi equation and Bellman equation on Wasserstein space, and the verification theorem for the ergodic mean field control problem. Relative to \cite{Sun-Yong-2024}, our derivation relies on analytic methods based on the dynamic programming principle. In particular, we formulate both the Hamilton--Jacobi equation and the Bellman equation on Wasserstein space for the control problems over the finite and infinite horizons, and we identify two Riccati systems, each with additional coupled equations, that characterize the optimal pairs and value functions. Most importantly, unlike \cite{Sun-Yong-2024}, where the turnpike property is derived from a homogeneous infinite-horizon control problem and a related static optimization problem, our formulation relies directly on the optimal pair $(\bar{X}(\cd), \bar{u}(\cd))$ of the ergodic control problem. This provides a natural and structurally consistent framework for comparing finite-horizon and ergodic mean field control problems. It is interesting to note that the desired processes used to establish turnpike properties in our work allow for arbitrary initial states, and we can recover the result in \cite{Sun-Yong-2024} by a special choice of the initial state of $\bar{X}(\cd)$, which is illustrated in Remark \ref{r:comparison}. Moreover, we establish the turnpike property at the level of value functions, see Lemma \ref{l:convergence_of_value_function}. Finally, we provide a verification theorem for the ergodic mean field control problem in terms of the Bellman equation, which not only facilitates the study of the mean field control problem under infinite horizon but also provides a benchmark in the linear-quadratic setting and offers insight into turnpike properties in the general mean field context.

The remainder of the paper is structured as follows: Section \ref{s:Finite_MFLQ} presents the solvability result for Problem (MFLQ)$^T$ using an analytical approach. We develop the associated Hamilton--Jacobi equation on Wasserstein space and Riccati system of differential equations, and then derive the optimal pair and value function for Problem (MFLQ)$^T$. In Section \ref{s:ergodic_control_problem_section}, 
we address the associated ergodic control problem. In particular, Subsection \ref{s:cell_problem} establishes the solvability of the Bellman equation and algebraic Riccati equations, while Subsection \ref{s:ergodic_control_problem} demonstrates that Problem (EC) provides a probabilistic interpretation to the Bellman equation and proves a verification theorem. In Subsection \ref{s:Estimates}, we derive key estimates linking the coefficient functions arising in Problem (MFLQ)$^T$ to the corresponding constant matrices and vectors in the Bellman equation. Then, leveraging these natural convergence results, we establish the turnpike properties between the optimal pairs of Problem (MFLQ)$^T$ and Problem (EC) in Subsection \ref{s:turnpike_properties}. Finally, Section \ref{s:numerical_experiments} provides numerical experiments illustrating the theoretical results of this paper.


\section{Finite-horizon LQ mean field control problem}
\label{s:Finite_MFLQ}

\ms

In this section, we establish the solvability of Problem (MFLQ)$^T$ by using an analytical approach. Throughout the discussion, we let $\cP_2(\dbR^k)$ denote the Wasserstein space of probability measures $\m$ on $\dbR^k$ satisfying $\int_{\dbR^k}|x|^2d\m(x)<\i$, endowed with the $2$-Wasserstein metric $\cW_2(\cd\,, \cd)$ defined by
$$\cW_2(\m_1,\m_2)=\inf_{\pi\in\Pi(\m_1,\m_2)}\Big(\int_{\dbR^k}\int_{\dbR^k}|x-y|^2
d\pi(x, y)\Big)^{\frac{1}{2}},$$
where $\Pi(\m_1,\m_2)$ is the collection of all probability measures on $\dbR^k\times\dbR^k$ with its marginals agreeing with $\m_1$ and $\m_2$, respectively. For each $x \in \dbR^k$, we denote by $\delta_{x}$ the Dirac measure concentrated at $x$. Moreover, we denote
\begin{equation*}
\cP^0_2(\dbR^k) = \Big\{\mu \in \cP_2(\dbR^k): \int_{\dbR^k} x \mu(dx) = 0\Big\}.
\end{equation*}
In the following subsection, we introduce the Hamiltonian and formulate the Hamilton--Jacobi equation on Wasserstein space associated with Problem (MFLQ)$^T$.

\subsection{Hamiltonian and the Hamilton--Jacobi equation}

\label{s:Hamiltonian_master_equation}

We introduce the centered state and control processes
$$Y(t) := X(t) - \dbE[X(t)], \quad \text{and} \quad Z(t) := u(t) - \dbE[u(t)].$$
We also set
\begin{equation*}
\begin{aligned}
& \h{A} = A + \bar{A}, \quad \h{B} = B + \bar{B}, \quad \h{C} = C + \bar{C}, \quad \h{D} = D + \bar{D}, \\
& \h{Q} = Q + \bar{Q}, \quad \h{R} = R + \bar{R}, \quad \h{S} = S + \bar{S}.
\end{aligned}
\end{equation*}
With this notion, the state equation \eqref{eq:state} can be written as
\begin{equation*}
\begin{cases}
\vspace{4pt}
\displaystyle
d X(t) = \big\{AY(t) + \h{A} \dbE[X(t)] + BZ(t) + \h{B} \dbE[u(t)] +  b \big\} dt \\
\vspace{4pt}
\displaystyle
\hspace{0.8in} + \big\{CY(t) + \h{C} \dbE[X(t)] + DZ(t) + \h{D} \dbE[u(t)] + \sigma \big\} d W(t), \quad t \ges 0, \\
X(0) = x.
\end{cases}
\end{equation*}
Taking expectation gives that
$$d \dbE[X(t)] = \big(\h{A} \dbE[X(t)] + \h{B} \dbE[u(t)] +  b \big) dt, \quad \dbE[X(0)] = x,$$
which yields the dynamic of the centered state process $Y(\cd)$ as follows:
\begin{equation*}
\begin{cases}
\vspace{4pt}
\displaystyle
d Y(t) = \big\{AY(t) + BZ(t) \big\} dt + \big\{CY(t) + \h{C} \dbE[X(t)] + DZ(t) + \h{D} \dbE[u(t)] + \sigma \big\} d W(t), \quad t \ges 0, \\
Y(0) = 0.
\end{cases}
\end{equation*}
Thus, the state process $X(\cd)$ in \eqref{eq:state} is uniquely determined by the pair $(\dbE[X(\cd)], Y(\cd))$. Moreover, a direct calculation shows that the cost functional \eqref{eq:cost_finite} can be equivalently rewritten as
\begin{equation*}
J_{T}(x; u(\cd)) = \dbE \Big[\int_0^T \wt{f}(Y(t), \dbE[X(t)], Z(t), \dbE[u(t)]) dt \Big],
\end{equation*}
where
\begin{equation}
\label{eq:f_tilde}
\wt{f}(y, \bar x, z, \bar{u}) = \left\lan \begin{pmatrix}
Q & S^\top \\ S & R
\end{pmatrix} \begin{pmatrix}
y \\ z
\end{pmatrix}, \begin{pmatrix}
y \\ z
\end{pmatrix} \right\ran + 2 \left\lan \begin{pmatrix}
\h{q} \\ \h{r}
\end{pmatrix}, \begin{pmatrix}
\bar{x} \\ \bar{u}
\end{pmatrix}\right\ran + \left\lan \begin{pmatrix}
\h{Q} & \h{S}^\top \\ \h{S} & \h{R}
\end{pmatrix} \begin{pmatrix}
\bar{x} \\ \bar{u}
\end{pmatrix}, \begin{pmatrix}
\bar{x} \\ \bar{u}
\end{pmatrix} \right\ran
\end{equation}
with $\h{q} = q + \bar{q}$ and $\h{r} = r + \bar{r}$. 
Consequently, we obtain
\begin{equation*}
\begin{aligned}
J_T(x; u(\cd)) & = \dbE \Big[\int_0^T \big\{\lan QY(t), Y(t) \ran + 2 \lan SY(t), Z(t) \ran + \lan RZ(t), Z(t) \ran \big\} dt \Big] \\
& \hspace{0.3in} + \int_0^T \big\{ \lan \h{Q} \dbE[X(t)], \dbE[X(t)] \ran + 2 \lan \h{S} \dbE[X(t)], \dbE[u(t)] \ran + \lan \h{R} \dbE[u(t)], \dbE[u(t)] \ran \\
& \hspace{1.0in} + 2\lan \h{q}, \dbE[X(t)] \ran + 2 \lan \h{r}, \dbE[u(t)] \ran \big\} dt.
\end{aligned}
\end{equation*}

Since both the state dynamics and the cost functional above depend on the distributions of $Y(\cd)$ and $Z(\cd)$, we consider the joint distribution of the centered state and control processes. To proceed, for each $t \in [0, T]$, let $\nu_t = \cL(Y(t))$ with $\int_{\dbR^n} y \nu_t(dy) = 0$, where $\cL(Y(t))$ is the law of $Y(t)$. Moreover, we denote $\mu_t := \cL(Y(t), Z(t))$ as the joint law of $Y(t)$ and $Z(t)$ for all $t \in [0, T]$. Thus, $\mu_t$ has the first marginal $\nu_t$ and since $Z(\cd)$ is centered, we have $\int_{\dbR^n \times \dbR^m} z \mu_t(dy, dz) = 0$. Accordingly, for a fixed $\nu \in \cP_2(\dbR^n)$ satisfying $\int_{\dbR^n} y \nu(dy) = 0$, i.e., $\nu \in \cP_2^0(\dbR^n)$, we define the admissible set of centered relaxed control as follows:
\begin{equation*}
\cA(\nu) := \Big\{\mu \in \cP_2(\dbR^n \times \dbR^m): \mu(\cd, \dbR^m) = \nu(\cd), \text{ and }\int_{\dbR^n \times \dbR^m} z \mu(dy, dz) = 0\Big\}.
\end{equation*}
Then, for the mean field control problem with the states $\dbE[X(\cd)]$ and $Y(\cd)$, by the dynamic programming approach (for example, see \cite{Carmona-Delarue-2018} and \cite{Pham-Wei-2018}), we consider the following Hamilton--Jacobi equation on the Wasserstein space: for $(t, \bar{x}, \nu) \in [0, T] \times \dbR^n \times \cP_2^0(\dbR^n)$,
\begin{equation}
\label{eq:Hamilton-Jacobi_equation}
\begin{cases}
\vspace{4pt}
\displaystyle
\partial_t V(t, \bar{x}, \nu) + \inf_{\bar{u} \in \dbR^m, \, \mu \in \cA(\nu)} \cH(t, \bar{x}, \nu, \bar{u}, \mu; V) = 0, \\
V(T, \bar{x}, \nu) = 0,
\end{cases}
\end{equation}
where, for a sufficiently smooth function $V(\cd,\cd,\cd)$, the Hamiltonian $\cH$ is defined as follows:
\begin{equation*}
\begin{aligned}
& \cH(t, \bar{x}, \nu, \bar{u}, \mu; V) \\ 
:= \ & \int_{\dbR^n \times \dbR^m} \Big\{\lan Ay + Bz, D_{\nu}V(t, \bar{x}, \nu)(y) \ran + \frac{1}{2} \lan D_yD_{\nu} V(t, \bar{x}, \nu)(y) \h{\sigma}(y, \bar{x}, z, \bar{u}), \h{\sigma}(y, \bar{x}, z, \bar{u}) \ran \\
& \hspace{0.3in} + \lan Qy, y\ran + 2 \lan Sy, z\ran + \lan Rz, z\ran \Big\} \mu(dy, dz) + \lan \h{A}\bar{x} + \h{B} \bar{u} + b, D_{\bar{x}} V(t, \bar{x}, \nu) \ran \\
& \hspace{0.3in} + \lan \h{Q}\bar{x}, \bar{x} \ran + 2 \lan \h{S}\bar{x}, \bar{u} \ran + \lan \h{R} \bar{u}, \bar{u} \ran + 2 \lan \h{q}, \bar{x} \ran + 2 \lan \h{r}, \bar{u} \ran.
\end{aligned}
\end{equation*}
Here, we denote $\h{\sigma}(y, \bar{x}, z, \bar{u}) := Cy + \h{C} \bar{x} + Dz + \h{D} \bar{u} + \sigma$. In the Hamilton--Jacobi equation \eqref{eq:Hamilton-Jacobi_equation}, we adopt the notion of $L$-derivative with respect to the measure variable, and the reader is referred to \cite{CDLL19} and \cite{Carmona-Delarue-2018} for the detailed definition. The solution $V(t, \bar{x}, \nu)$ to the equation \eqref{eq:Hamilton-Jacobi_equation} can be interpreted as the value function of the mean field control problem starting at an arbitrary time $t \in [0, T]$, with 
the initial conditions $\dbE[X(t)] = \bar{x}$ and $\cL(Y(t)) = \nu$.

\subsection{Riccati system and solvability of Problem (MFLQ)$^T$}
\label{s:Riccati_system_finite}

In the previous sections, we have formulated Problem (MFLQ)$^T$ and derived the Hamilton--Jacobi equation on the Wasserstein space \eqref{eq:Hamilton-Jacobi_equation} associated with Problem (MFLQ)$^T$. We now define the notion of closed-loop solvability for Problem (MFLQ)$^T$ and proceed to establish its solvability.

\ms

Let $\hTh[0,T]=L^\i(0,T;\dbR^{m\times n})$. For any $(\Th(\cd), \bar \Th(\cd), \theta(\cd)) \in \hTh[0,T] \times \hTh[0,T] \times \sU[0,T]$, let
$$u(t) = \Th(t) (X(t) - \dbE[X(t)]) + \bar \Th(t) \dbE[X(t)] + \theta(t), \quad t \in [0,T],$$
with $X(\cd) := X(\cd\,; x,\Th(\cd), \bar \Th(\cd), \theta(\cd))$ being the solution to the following closed-loop system
\begin{equation}
\label{eq:closed-loop-finite}
\begin{cases}
\vspace{4pt}
\displaystyle
d X(t) = \big\{(A+B\Theta(t)) X(t) + \big(\bar A + \h{B}\bar \Theta(t) - B \Theta(t) \big) \dbE[X(t)] + B \theta(t) + \bar{B} \dbE[\theta(t)] + b \big\} dt \\
\vspace{4pt}
\displaystyle
\hspace{0.7in} + \big\{(C+D\Theta(t)) X(t) + \big(\bar C + \h{D}\bar \Theta(t) - D \Theta(t)\big) \dbE[X(t)] + D \theta(t) \\
\vspace{4pt}
\displaystyle
\hspace{0.9in} + \bar{D} \dbE[\theta(t)] + \sigma \big\} dW(t), \quad t \in [0, T], \\
X(0) = x,
\end{cases}
\end{equation}
We adopt the following notation:
$$J_T(x;\Th(\cd), \bar \Th(\cd), \theta(\cd)) = J_T \big(x; \Th(\cd)(X(\cd) - \dbE[X(\cd)]) + \bar \Th(\cd) \dbE[X(\cd)] + \theta(\cd)\big).$$
A tuple $(\Th_T^{*}(\cd), \bar \Th_{T}^{*} (\cd), \th_T^{*}(\cd))$ is called a closed-loop optimal strategy if it satisfies
$$J_T(x; \Th_T^{*}(\cd), \bar \Th_{T}^{*} (\cd), \th_T^{*}(\cd)) \les J_T(x; \Th(\cd), \bar \Th(\cd), \theta(\cd))$$
for all $(\Th(\cd), \bar \Th(\cd), \theta(\cd)) \in \hTh[0,T] \times \hTh[0,T] \times \sU[0,T]$ and $x \in \dbR^n.$
When such a tuple $(\Th_T^{*}(\cd), \bar \Th_{T}^{*} (\cd), \th_T^{*}(\cd))$ exists, we say that Problem (MFLQ)$^T$ is closed-loop solvable on $[0,T]$.

\ms

Next, we investigate the solvability of Problem (MFLQ)$^T$. Let $C^1([0,T]; \dbR^{n\times m})$ denote the space of $\dbR^{n\times m}$-valued continuous functions with continuous first order derivative on $[0,T]$. We consider the ansatz that the solution $V(\cd, \cd, \cd)$ to the Hamilton--Jacobi equation \eqref{eq:Hamilton-Jacobi_equation} takes the following form:
\begin{equation}
\label{eq:value_function_finite}
V(t, \bar{x}, \nu) = \int_{\dbR^n} \lan P_T(t) y, y \ran \nu(dy) + \lan \Pi_T(t) \bar{x}, \bar{x} \ran + 2 \lan p_T(t), \bar{x} \ran + \kappa_T(t)
\end{equation}
where $P_T(\cd), \ \Pi_T(\cd) \in C^1([0, T]; \dbS^n)$, $p_T(\cd) \in C^1([0, T], \dbR^n)$ and $\kappa_T(\cd) \in C^1([0, T]; \dbR)$. For the above ansatz, it is clear that
$$\partial_t V(t, \bar{x}, \nu) = \int_{\dbR^n} \lan \dot{P}_T(t) y, y \ran \nu(dy) + \lan \dot{\Pi}_T(t) \bar{x}, \bar{x} \ran + 2 \lan \dot{p}_T(t), \bar{x} \ran + \dot{\kappa}_T(t)$$
and
$$D_{\nu}V(t, \bar{x}, \nu)(y) = 2 P_T(t) y, \quad D_y D_{\nu} V(t, \bar{x}, \nu)(y) = 2P_T(t), \quad D_{\bar{x}} V(t, \bar{x}, \nu) = 2 \Pi_T(t) \bar{x} + 2 p_T(t)$$
for all $t \in [0, T]$, $\bar{x} \in \dbR^n$ and $\nu \in \cP_2^0(\dbR^n)$. Since $\int_{\dbR^n \times \dbR^m} y \mu(dy, dz) = 0$ and $\int_{\dbR^n \times \dbR^m} z \mu(dy, dz) = 0$, we have
$$\int_{\dbR^n \times \dbR^m} \lan Cy + Dz, P_T(t) (\h{C}\bar{x} + \h{D} \bar{u} + \sigma) \ran \mu(dy, dz) = 0.$$
It follows that
\begin{equation*}
\begin{aligned}
& \int_{\dbR^n \times \dbR^m} \frac{1}{2} \lan D_y D_{\nu} V(t, \bar{x}, \nu)(y) \h{\sigma}(y, \bar{x}, z, \bar{u}), \h{\sigma}(y, \bar{x}, z, \bar{u}) \ran \mu(dy, dz) \\
= \ & \int_{\dbR^n \times \dbR^m} \lan P_T(t)(Cy + Dz), Cy + Dz \ran \mu(dy, dz) + \lan P_T(t)(\h{C}\bar{x} + \h{D} \bar{u} + \sigma), \h{C}\bar{x} + \h{D} \bar{u} + \sigma \ran.
\end{aligned}
\end{equation*}
This decomposition plays a key role in deriving the minimizer of the Hamiltonian $\cH$. 
For $P, \Pi \in \dbS^n$, we define the following terms
\begin{equation*}
\begin{aligned}
& \mathcal{Q}(P) = P A+A^\top P + C^\top P C + Q, \quad \h{\mathcal Q}(P, \Pi) = \Pi \h A+\h A^\top \Pi + \h C^\top P \h C + \h Q, \\
& \mathcal S(P) = B^\top P + D^\top P C+ S, \quad \h{\mathcal S}(P, \Pi) = \h{B}^\top \Pi + \h{D}^\top P \h{C} + \h{S}, \\
& \mathcal R(P) = R + D^\top P D, \quad \h{\cR}(P) = \h{R} + \h{D}^\top P \h{D}.
\end{aligned}
\end{equation*}
Then, substituting the ansatz \eqref{eq:value_function_finite} into the Hamilton--Jacobi equation \eqref{eq:Hamilton-Jacobi_equation}, we obtain
\begin{equation*}
\begin{aligned}
& \cH(t, \bar{x}, \nu, \bar{u}, \mu; V) \\
= \ & \int_{\dbR^n \times \dbR^m} \big\{\lan \cQ(P_T(t))y, y \ran + 2 \lan \cS(P_T(t)) y, z \ran + \lan \cR(P_T(t)) z, z \ran \big\} \mu(dy, dz)  \\
& \hspace{0.3in} + \lan \h{\cQ}(P_T(t), \Pi_T(t)) \bar{x}, \bar{x} \ran + 2 \lan \h{\cS}(P_T(t), \Pi_T(t)) \bar{x}, \bar{u} \ran + \lan \h{\cR}(P_T(t)) \bar{u}, \bar{u} \ran \\
& \hspace{0.3in} + 2 \lan \h{A}^\top p_T(t) + \Pi_T(t) b + \h{C}^\top P_T(t) \sigma + \h{q}, \bar{x} \ran + 2 \lan \h{B}^\top p_T(t) + \h{D}^\top P_T(t) \sigma + \h{r}, \bar{u} \ran \\
& \hspace{0.3in} + 2 \lan p_T(t), b \ran + \lan P_T(t) \sigma, \sigma \ran \\
:= \ & \cH_c(t, \nu, \mu) + \cH_m(t, \bar{x}, \bar{u}),
\end{aligned}
\end{equation*}
where
$$\cH_c(t, \nu, \mu) := \int_{\dbR^n \times \dbR^m} \big\{\lan \cQ(P_T(t))y, y \ran + 2 \lan \cS(P_T(t)) y, z \ran + \lan \cR(P_T(t)) z, z \ran \big\} \mu(dy, dz).$$
To proceed, we define
\begin{equation}
\label{eq:theta_T_star}
\begin{aligned}
& \Th_T^{*}(t) = - \cR(P_T(t))^{-1} \cS(P_T(t)), \\
& \bar \Th_T^{*}(t) = - \h{\cR}(P_T(t))^{-1} \h \cS(P_T(t), \Pi_T(t)), \\
& \theta_T^{*}(t) = - \h{\cR}(P_T(t))^{-1} (\h{B}^\top p_T(t) + \h{D}^\top P_T(t) \sigma + \h{r}).
\end{aligned}
\end{equation}
Note that, if $\cR(P_T(t)) > 0$ for all $t \in [0, T]$, the integrand of $\cH_c(t, \nu, \mu)$ is strictly convex with respect to $z$. Thus, for any $\nu \in \cP^0_2(\dbR^n)$ and $y \in \dbR^n$, we can choose $z^*(t, y) = \Theta_T^*(t) y$ such that the integral part of $\cH$ is minimized, i.e.,
$$\inf_{\mu \in \mathcal{A}(\nu)} \cH_c(t, \nu, \mu) = \int_{\dbR^n} \lan (\cQ(P_T(t)) - \cS(P_T(t))^\top \cR(P_T(t))^{-1} \cS(P_T(t)))y, y \ran \nu(dy).$$
The corresponding minimizing measure $\mu^*$ is given by $\mu^*(dy, dz) = \nu(dy) \delta_{z^*(t, y)}(dz)$, where $\delta_{z^*(t,y)}$ denotes the Dirac measure at $z^*(t,y)$. Since $\int_{\dbR^n} y \nu(dy) = 0$, it follows that $\int_{\dbR^n \times \dbR^m} z \mu^*(dy, dz) = \int_{\dbR^n} z^*(t, y) \nu(dy) = 0$ for all $t \in [0, T]$. Similarly, if $\h{\cR}(P_T(t)) > 0$ for all $t \in [0, T]$, let $\bar{u}^*(t, \bar{x}) = \bar{\Theta}^*_T(t) \bar{x} + \theta_T^*(t)$, the remaining terms of $\cH$ are minimized, i.e.,
\begin{equation*}
\begin{aligned}
& \inf_{\bar{u} \in \dbR^m} \cH_m(t, \bar{x}, \bar{u}) \\
= \ & \lan (\h{\cQ}(P_T(t), \Pi_T(t)) - \h{\cS}(P_T(t), \Pi_T(t))^\top \h{\cR}(P_T(t))^{-1} \h{\cS}(P_T(t), \Pi_T(t))) \bar{x}, \bar{x} \ran \\
& \hspace{0.3in} + 2 \lan (\h{A} + \h{B}\bar{\Th}_T^*(t))^\top p_T(t) + (\h{C} + \h{D} \bar{\Th}_T^*(t))^\top P_T(t) \sigma + \bar{\Th}_T^*(t)^\top  \h{r} + \Pi_T(t) b + \h{q}, \bar{x} \ran \\
& \hspace{0.3in} + 2 \lan p_T(t), b \ran + \lan P_T(t) \sigma, \sigma \ran - \lan \h{R}(P_T(t)) \theta_T^*(t), \theta_T^*(t) \ran.
\end{aligned}
\end{equation*}
Hence, for all $(t, y, \bar x) \in [0, T] \times \dbR^n \times \dbR^n$, the feedback form of the optimal control is given as follows
\begin{equation*}
\begin{aligned}
\bar u_T(t, y, \bar x) &= \Theta_T^*(t) y + \bar{\Theta}_T^*(t) \bar{x} + \theta_T^*(t).
\end{aligned}
\end{equation*}

With the ansatz of solution in \eqref{eq:value_function_finite} and the above results, by isolating the terms associated with $y$ and $\bar x$, the Hamilton--Jacobi equation \eqref{eq:Hamilton-Jacobi_equation} leads to the following system of ordinary differential equations (ODEs):
\begin{equation}
\label{eq:Riccati_ODE}
\begin{cases}
\vspace{4pt}
\displaystyle
\dot{P}_T(t) + \cQ(P_T(t)) - \cS(P_T(t))^\top \cR(P_T(t))^{-1} \cS(P_T(t)) = 0, \\
\vspace{4pt}
\displaystyle
\dot{\Pi}_T(t) + \h{\cQ}(P_T(t), \Pi_T(t)) - \h{\cS}(P_T(t), \Pi_T(t))^\top \h{\cR}(P_T(t))^{-1} \h{\cS}(P_T(t), \Pi_T(t)) = 0, \\
\vspace{4pt}
\displaystyle
\dot{p}_T(t) + (\h{A} + \h{B}\bar{\Th}_T^*(t))^\top p_T(t) + (\h{C} + \h{D} \bar{\Th}_T^*(t))^\top P_T(t) \sigma + \bar{\Th}_T^*(t)^\top  \h{r} + \Pi_T(t) b + \h{q} = 0, \\
\vspace{4pt}
\displaystyle
\dot{\kappa}_T(t) + 2 \lan p_T(t), b \ran + \lan P_T(t) \sigma, \sigma \ran - \lan \h{R}(P_T(t)) \theta_T^*(t), \theta_T^*(t) \ran = 0,
\end{cases}
\end{equation}
with the terminal conditions
$$P_T(T) = \Pi_T(T) = 0 \in \dbS^n, \quad p_T(T) = 0 \in \dbR^n, \quad \kappa_T(T) = 0 \in \dbR.$$

\begin{remark}
\textit{The first two equations for $P_T(\cd)$ and $\Pi_T(\cd)$ in the system \eqref{eq:Riccati_ODE} are standard Riccati differential equations and they were also derived in \cite{Yong13} and \cite{Sun-Yong-2024} via the stochastic maximum principle. In contrast, our analysis is based on the dynamic programming principle and incorporates additional equations that jointly characterize the optimal control and the value function for the finite-horizon mean field control problem. This extended formulation enables a direct comparison with the algebraic Riccati system arising in the ergodic control problem and plays a crucial role in establishing the turnpike property for both the optimal pair and the value function.}
\end{remark}

Next, to ensure the solvability of the system of ODEs \eqref{eq:Riccati_ODE} and of Problem (MFLQ)$^T$, we impose the following hypothesis. 

\ms 

{\bf(H1)} \textit{The coefficient matrices in \eqref{eq:running_cost} satisfy
$$R, \h{R} \in \dbS^m_{++}, \quad Q - S^\top R^{-1} S \in \dbS^{n}_{++}, \quad \h{Q} - \h{S}^\top \h{R}^{-1} \h{S} \in \dbS^{n}_{++}.$$
}

For Problem (MFLQ)$^T$, we present the following results. Throughout, $I_n$ denotes the $n \times n$ identity matrix.

\begin{lemma}
\label{l:finite_time_control} 
Let {\rm(H1)} hold. Then

\ms

{\rm(i)} The system of ODEs \eqref{eq:Riccati_ODE} for $\{P_T(t), \Pi_T(t), p_T(t), \kappa_T(t): t \in [0, T]\}$ admits a unique solution such that $P_T(t) \ges 0$ and $\Pi_T(t) \ges 0$ for all $t \in [0, T]$. Moreover, there exists a constant $\delta > 0$ such that $\cR(P_T(t)) \ges \d I_m$ and $\h{\cR}(P_T(t)) \ges \delta I_m$ for all $t \in [0, T]$.

{\rm(ii)} For each initial $x \in \dbR^n$, Problem (MFLQ)$^T$ admits a unique open-loop optimal control $u_T(\cd)$ with the following closed-loop representation:
\begin{equation}
\label{eq:optimal_control_finite}
u_T(t) = \Th_T^*(t) (X_T(t) -\dbE[X_T(t)]) + \bar \Th_T^*(t) \dbE[X_T(t)] + \th_T^*(t), \quad t\in[0,T],
\end{equation}
where $X_T(\cd):= X(\cd\,; x, \Th_T^*(\cd), \bar \Th_T^*(\cd), \th_T^*(\cd))$ is the solution to the corresponding closed-loop system (similar to \eqref{eq:closed-loop-finite}), with
the tuple $(\Th_T^{*}(\cd), \bar \Th_{T}^{*} (\cd), \th_T^{*}(\cd))$ is given by \eqref{eq:theta_T_star}. 

\ms

{\rm(iii)} The value function $V_T(\cd)$ in \eqref{eq:MFLQ_finite} is expressed as
\begin{equation*}
\begin{aligned}
V_T(x) &= \lan \Pi_T(0) x, x \ran + 2 \lan p_T(0), x \ran + \kappa_T(0) \\
&= \lan \Pi_T(0) x, x \ran + 2 \lan p_T(0), x \ran + \int_0^T \big(\lan P_T(t) \sigma, \sigma \ran + 2 \lan p_T(t), b \ran  - \lan \h{\cR}(P_T(t)) \theta_T^*(t), \theta_T^*(t) \ran \big) dt,
\end{aligned}
\end{equation*}
where $\{P_T(t), \Pi_T(t), p_T(t), \kappa_T(t): t \in [0, T]\}$ is the solution to the system of equations \eqref{eq:Riccati_ODE}, and $\theta_T^*(\cd)$ is defined in \eqref{eq:theta_T_star}.
\end{lemma}
\begin{proof}
{\rm(i)} Under assumption (H1), the unique solvability of the system of equations $\{P_T(t), \Pi_T(t): t \in [0, T]\}$ in \eqref{eq:Riccati_ODE} follows directly from the results in Section 4 of \cite{Yong13} or Lemma 2.1 in \cite{Sun-Yong-2024}. Moreover, $P_T(t)$, $\Pi_T(t) \in \dbS_{+}^{n}$ for all $t \in [0, T]$. Since $R$ and $\h{R}$ are both positive definite, there exists a constant $\delta > 0$ such that $\cR(P_T(t))$ and $\h{\cR}(P_T(t)) \ges \delta I_m$ for all $t \in [0, T]$. Given $P_T(\cd)$, $\Pi_T(\cd) \in C^1([0, T]; \dbS^n_{+})$, the equation for $p_T(\cd)$ is a linear ODE with a terminal condition and thus admits a unique solution by the standard theory of linear ODEs. Finally, once $P_T(\cd)$ and $p_T(\cd)$ have been determined, $\kappa_T(\cd)$ is uniquely obtained by direct integrating its scalar ODE.
\ms

{\rm(ii)} The results follow directly from a standard verification theorem for mean field control problems; see also Theorem 5.2 in \cite{Sun17}.
\ms

{\rm(iii)} By the definition of the value function for (MFLQ)$^T$ in \eqref{eq:MFLQ_finite} and employing an argument analogous to that in Theorem 5.2 of \cite{Sun17}, we obtain
\begin{equation*}
\begin{aligned}
V_T(x) & = V(0, x, \delta_0) = \int_{\dbR^n} \lan P_T(0) y, y \ran \delta_0(dy) + \lan \Pi_T(0) x, x \ran + 2 \lan p_T(0), x \ran + \kappa_T(0) \\
&= \lan \Pi_T(0) x, x \ran + 2 \lan p_T(0), x \ran + \kappa_T(0) 
\end{aligned}
\end{equation*}
since $\dbE[X_{T}(0)] = x$ and $\cL(X_T(0) - \dbE[X_T(0)]) = \delta_0$, where $\delta_0$ denotes the Dirac measure concentrated at $0 \in \dbR^n$. Here, $V(\cd, \cd, \cd)$ in \eqref{eq:value_function_finite} solves the Hamilton--Jacobi equation \eqref{eq:Hamilton-Jacobi_equation}, and $\{\Pi_T(t), p_T(t), \kappa_T(t): t \in [0, T]\}$ is the solution to the system of equations \eqref{eq:Riccati_ODE}.
\end{proof}


\section{Ergodic control problem}
\label{s:ergodic_control_problem_section}

In the previous sections, we introduced the finite-time horizon LQ mean field control problem and demonstrated the solvability of Problem (MFLQ)$^T$ through an analytical approach. A natural extension of this problem is to consider the case where the time horizon is infinite, i.e., the interval $[0, \infty)$. Specifically, we now examine the same state equation \eqref{eq:state}, but with the cost functional defined as
\begin{equation*}
J_{\infty}(x; u(\cd)) = \dbE \Big[\int_0^{\infty} f(X(t), \dbE[X(t)], u(t), \dbE[u(t)]) dt \Big].
\end{equation*}
Note that, while the non-homogeneous state equation \eqref{eq:state} on $[0, \infty)$ may be solvable if $b$ or $\sigma$ is non-zero, the cost functional defined above may not be well-defined, regardless of whether the homogeneous system is stabilizable. Consequently, the infinite-time version of the LQ mean field control problem, denoted as (MFLQ)$^\infty$, is not well-posed. 

To address this issue, we turn our attention to the Bellman equation in Subsection \ref{s:cell_problem}, following the approach presented in \cite{Jian-Jin-Song-Yong-2024}, where it is referred to as the cell problem. We then demonstrate that Problem (EC) provides a probabilistic interpretation of the Bellman equation and prove a verification theorem in Subsection \ref{s:ergodic_control_problem}, establishing that the solution to Problem (EC) can be derived from the solution to the Bellman equation.

\subsection{The Bellman equation and algebraic Riccati equations}
\label{s:cell_problem}

Before introducing the Bellman equation, we first examine the homogeneous state equation on the infinite horizon $[0, \i)$ and impose a stabilizability condition for this system.

\subsubsection{Homogeneous state equation and stabilizability condition}

The homogeneous state process, denoted compactly as $[A, \bar A, C, \bar C; B, \bar{B}, D, \bar{D}]$, is defined as follows ($b = \sigma = 0$, compare with 
\eqref{eq:state}):
\begin{equation*}
\begin{cases}
\vspace{4pt}
\displaystyle
d X(t) = \{AX(t) + \bar A \dbE[X(t)] + Bu(t) + \bar{B} \dbE[u(t)]\} dt \\
\vspace{4pt}
\displaystyle
\hspace{0.8in} + \{CX(t) + \bar C \dbE[X(t)] + Du(t) + \bar{D} \dbE[u(t)] \} d W(t), \quad t \ges 0, \\
X(0) = x.
\end{cases}
\end{equation*}

\begin{definition}
\textit{The homogeneous system $[A, \bar A, C, \bar C; B, \bar{B}, D, \bar{D}]$ is MF-$L^2$-stabilizable if there exists a pair $(\Theta, \bar \Theta) \in \mathbb R^{m \times n} \times \mathbb R^{m \times n}$, called the MF-$L^2$-stabilizer of $[A, \bar A, C, \bar C; B, \bar{B}, D, \bar{D}]$, such that if $X(\cdot) := X(\cd \,; x, \Theta, \bar \Theta)$ is the solution to the following homogeneous closed-loop system
\begin{equation*}
\begin{cases}
\vspace{4pt}
\displaystyle
d X(t) = \big\{(A+B\Theta) X(t) + \left[\bar A + \bar{B} \bar{\Theta} + B(\bar \Theta - \Theta)\right] \dbE[X(t)] \big\} dt \\
\vspace{4pt}
\displaystyle
\hspace{0.8in} + \big\{(C+D\Theta) X(t) + \left[\bar C + \bar{D} \bar{\Theta} + D(\bar \Theta - \Theta)\right] \dbE[X(t)] \big\} dW(t), \quad t \ges 0, \\
X(0) = x,
\end{cases}
\end{equation*}
and 
\begin{equation*}
u(t) = \Theta (X(t) - \dbE[X(t)]) + \bar \Theta \dbE[X(t)], \quad t \ges 0,
\end{equation*}
then
\begin{equation}
\label{eq:integrability}
\dbE \Big[\int_0^{\infty} \big( |X(t)|^2 + |u(t)|^2 \big) dt \Big] < \infty.
\end{equation}}
\end{definition}

\ms

To ensure that the homogeneous system $[A, \bar A, C, \bar C; B, \bar{B}, D, \bar{D}]$ is MF-$L^2$-stabilizable, we introduce the following assumption.

\ms 

{\bf(H2)} 
\textit{The controlled ordinary differential equation
\begin{equation}
\label{eq:homo_ode}
\dot{X}(t) = \h{A} X(t) + \h{B} u(t), \quad t \ges 0
\end{equation}
is stabilizable, i.e., there exists a matrix $\bar \Th \in \dbR^{m \times n}$ such that all the eigenvalues of $\h{A} + \h{B} \bar \Theta$ have negative real parts. In this case, we call $\bar \Th$ a stabilizer of \eqref{eq:homo_ode}. Moreover, the controlled SDE
\begin{equation}
\label{eq:homo_sde}
d X(t) = (AX(t) + Bu(t)) dt + (CX(t) + Du(t)) dW(t), \quad t \ges 0
\end{equation}
is $L^2$-stabilizable, i.e., there exists a matrix $\Th \in \dbR^{m \times n}$ such that for any initial state $x \in \dbR^{n}$, the solution $X(\cd)$ to
\begin{equation*}
\begin{cases}
\vspace{4pt}
\displaystyle
d X(t) = (A + B\Th) X(t) dt + (C+ D \Th)X(t) dW(t), \quad t \ges 0 \\
X(0) = x
\end{cases}
\end{equation*}
satisfies $\dbE[\int_0^\infty |X(t)|^2 dt] < \infty$, i.e., $X(\cd) \in \sX[0, \infty)$.
}
\ms

From Corollary 4.4 (ii) in \cite{HLY15}, under the assumption (H2), the system $[A, \bar A, C, \bar C; B, \bar{B}, D, \bar{D}]$ is MF-$L^2$-stabilizable. Consequently, the following set is nonempty
\begin{equation}
\label{eq:u_ad_infinity}
\sU_{ad}[0,\i) := \big\{u(\cd) \in \sU[0,\i) \bigm| X(\cd\,; x, u(\cd)) \in L^2_\dbF(0,\i; \dbR^n) \big\}.
\end{equation}

\subsubsection{The Bellman equation and its solvability}

To establish the solvability of the ergodic control problem \eqref{eq:ergodic_control_problem}, following the approach developed in Subsection \ref{s:Hamiltonian_master_equation}, we first consider the following Bellman equation, which is also referred to as the cell problem in \cite{Tran2021,Jian-Jin-Song-Yong-2024}:

\ms

{\bf Problem (C)}. Determine a pair $(V(\cd, \cd), c_0)$ such that the Bellman equation
\begin{equation}
\label{eq:cell_problem}
c_0 = \inf_{\bar{u} \in \dbR^m, \, \mu \in \mathcal A(\nu)} \cH_{\infty}(\bar{x}, \nu, \bar{u}, \mu; V),
\end{equation}
holds in the strong sense for all $\bar x \in \dbR^n$ and $\nu \in \cP_2^0(\dbR^n)$, where
\begin{equation}
\label{eq:Hamiltonian_infty}
\begin{aligned}
& \cH_{\infty}(\bar{x}, \nu, \bar{u}, \mu; V) \\ 
:= \ & \int_{\dbR^n \times \dbR^m} \Big\{\lan Ay + Bz, D_{\nu}V(\bar{x}, \nu)(y) \ran + \frac{1}{2} \lan D_y D_{\nu} V(\bar{x}, \nu)(y) \h{\sigma}(y, \bar{x}, z, \bar{u}), \h{\sigma}(y, \bar{x}, z, \bar{u}) \ran \\
& \hspace{0.3in} + \lan Qy, y\ran + 2 \lan Sy, z\ran + \lan Rz, z\ran \Big\} \mu(dy, dz) + \lan \h{A}\bar{x} + \h{B} \bar{u} + b, D_{\bar{x}} V(\bar{x}, \nu) \ran \\
& \hspace{0.3in} + \lan \h{Q}\bar{x}, \bar{x} \ran + 2 \lan \h{S}\bar{x}, \bar{u} \ran + \lan \h{R} \bar{u}, \bar{u} \ran + 2 \lan \h{q}, \bar{x} \ran + 2 \lan \h{r}, \bar{u} \ran.
\end{aligned}
\end{equation}


Next, we establish the solvability for Problem (C) by leveraging the results for the corresponding system of algebraic equations. To proceed,  we denote
\begin{equation}
\label{eq:theta_star}
\begin{aligned}
\Th^{*} &= - \cR(P)^{-1} \cS(P), \\
\bar \Th^{*} & = - \h{\cR}(P)^{-1} \h{\cS}(P, \Pi), \\
\theta^{*} &= - \h{\cR}(P)^{-1} (\h{B}^\top p + \h{D}^\top P \sigma + \h{r}),
\end{aligned}
\end{equation}
and introduce the following system of algebraic equations
\begin{equation}
\label{eq:algebraic_Riccati}
\begin{cases}
\vspace{4pt}
\displaystyle
\cQ(P) - \cS(P)^\top \cR(P)^{-1} \cS(P) = 0, \\
\vspace{4pt}
\displaystyle
\h{\cQ}(P, \Pi) - \h{\cS}(P, \Pi)^\top \h{\cR}(P)^{-1} \h{\cS}(P, \Pi) = 0, \\
\vspace{4pt}
\displaystyle
(\h{A} + \h{B}\bar{\Th}^*)^\top p + (\h{C} + \h{D} \bar{\Th}^*)^\top P \sigma + (\bar{\Th}^*)^\top \h{r} + \Pi b + \h{q} = 0, \\
2 p^\top b + \sigma^\top P \sigma - \lan \h{R}(P) \theta^*, \theta^* \ran = c_0.
\end{cases}
\end{equation}

Under assumptions (H1)--(H2), we establish the following results concerning the solvability of the algebraic system \eqref{eq:algebraic_Riccati} and the Bellman equation \eqref{eq:cell_problem}.

\begin{lemma}
\label{l:infinite_time_homo} 
Let {\rm(H1)--(H2)} hold. Then,

\ms

{\rm(i)} The system of algebraic equations \eqref{eq:algebraic_Riccati} for $\{P, \Pi, p, c_0\}$ admits a unique solution such that $P \in \dbS^n_{++}$ and $\Pi \in \dbS^n_{++}$. Moreover, the matrix 
$$\bar{\Th}^{*} = - \h{\cR}(P)^{-1} \h \cS(P, \Pi)$$
is a stabilizer of \eqref{eq:homo_ode}, and the matrix 
$$\Th^{*} = - \cR(P)^{-1} \cS(P)$$
is a stabilizer of \eqref{eq:homo_sde}.

\ms

{\rm(ii)} The Bellman equation \eqref{eq:cell_problem} admits a solution $(V(\cd, \cd), c_0)$ with $V(\bar{x}, \nu)$ taking the following form
\begin{equation}
\label{eq:value_function_cell}
V(\bar{x}, \nu) = \int_{\dbR^n} \lan P y, y \ran \nu(dy) + \lan \Pi \bar{x}, \bar{x} \ran + 2 \lan p, \bar{x} \ran
\end{equation}
where $\{P, \Pi, p\}$ is the unique solution to \eqref{eq:algebraic_Riccati}, and
\begin{equation}
\label{eq:c_0}
c_0 = 2 p^\top b + \sigma^\top P \sigma - \lan \h{R}(P) \theta^*, \theta^* \ran.
\end{equation}
\end{lemma}

\begin{proof}
{\rm(i)} From Theorem 3.4.1 in \cite{Sun-Yong-2020} or Lemma 2.2 in \cite{Sun-Yong-2024}, we obtain the unique solvability of $P$ and $\Pi$ in \eqref{eq:algebraic_Riccati}. Moreover, both $P$ and $\Pi$ are positive definite, and $\bar \Th^*$ and $\Th^*$ serve as the stabilizers of \eqref{eq:homo_ode} and \eqref{eq:homo_sde}, respectively. Thus, the matrices $A + B \Th^*$ and $\h{A} + \h{B} \bar{\Th}^*$ are stable, that is, these two matrices are Hurwitz with all eigenvalues having strictly negative real parts. With the given $P$ and $\Pi$, the third equation in \eqref{eq:algebraic_Riccati} becomes a linear algebraic equation for $p$. Since the matrix $\h{A} + \h{B} \bar{\Th}^*$ is stable, it is invertible, and hence this equation admits a unique solution $p$. Finally, once $P$ and $p$ are given, $c_0$ is uniquely determined by the last equation in \eqref{eq:algebraic_Riccati}.
\ms

{\rm(ii)} Let $\{P, \Pi, p\}$ be the unique solution to \eqref{eq:algebraic_Riccati}. From the explicit form of $V(\cd, \cd)$ in \eqref{eq:value_function_cell}, we have
$$D_{\nu}V(\bar{x}, \nu)(y) = 2Py, \quad D_y D_{\nu}V(\bar{x}, \nu)(y) = 2P, \quad D_{\bar{x}} V(\bar{x}, \nu) = 2 \Pi \bar{x} + 2p$$
for all $\bar{x}, y \in \dbR^n$ and $\nu \in \cP_2^0(\dbR^n)$. Proceeding with calculations analogous to those in Section \ref{s:Riccati_system_finite}, we obtain
\begin{equation*}
\begin{aligned}
\cH_{\infty}(\bar{x}, \nu, \bar{u}, \mu; V) & = \int_{\dbR^n \times \dbR^m} \big\{\lan \cQ(P) y, y \ran + 2 \lan \cS(P) y, z \ran + \lan \cR(P) z, z \ran \big\} \mu(dy, dz)  \\
& \hspace{0.3in} + \lan \h{\cQ}(P, \Pi) \bar{x}, \bar{x} \ran + 2 \lan \h{\cS}(P, \Pi) \bar{x}, \bar{u} \ran + \lan \h{\cR}(P) \bar{u}, \bar{u} \ran \\
& \hspace{0.3in} + 2 \lan \h{A}^\top p + \Pi b + \h{C}^\top P \sigma + \h{q}, \bar{x} \ran + 2 \lan \h{B}^\top p + \h{D}^\top P \sigma + \h{r}, \bar{u} \ran \\
& \hspace{0.3in} + 2 \lan p, b \ran + \lan P \sigma, \sigma \ran.
\end{aligned}
\end{equation*}
Hence, for any $\bar{x} \in \dbR^n$ and $\nu \in \cP_2^0(\dbR^n)$, we can choose $z^*(y) = \Theta^* y$ and $\bar{u}^*(\bar{x}) = \bar{\Theta}^* \bar{x} + \theta^*$ such that $\cH_{\infty}(\bar{x}, \nu, \bar{u}, \mu; V)$ is minimized, where $\Theta^*, \bar{\Theta}^*$ and $\theta^*$ are defined in \eqref{eq:theta_star}. The corresponding minimizing measure $\mu^*$ is given by $\mu^*(dy, dz) = \nu(dy) \delta_{z^*(y)}(dz)$. Consequently, the feedback form for the optimal control is obtained and it is given by
\begin{equation}
\label{eq:feedback_form}
\bar u_{\infty} (y, \bar x) = z^*(y) + \bar{u}^*(\bar{x}) = \Th^{*} y + \bar \Th^{*} \bar x + \theta^{*}
\end{equation}
for all $(y, \bar x) \in \dbR^n \times \dbR^n$. Then, from the ansatz in \eqref{eq:value_function_cell} and the feedback function in \eqref{eq:feedback_form}, the Bellman equation \eqref{eq:cell_problem} can be reduced to
\begin{equation*}
\begin{aligned}
c_0 &= \int_{\dbR^n} \lan (\cQ(P) - \cS(P)^\top \cR(P)^{-1} \cS(P))y, y \ran \nu(dy) \\
& \hspace{0.3in} + \lan (\h{\cQ}(P, \Pi) - \h{\cS}(P, \Pi)^\top \h{\cR}(P)^{-1} \h{\cS}(P, \Pi)) \bar{x}, \bar{x} \ran \\
& \hspace{0.3in} + 2 \lan (\h{A} + \h{B}\bar{\Th}^*)^\top p + (\h{C} + \h{D} \bar{\Th}^*)^\top P \sigma + (\bar{\Th}^*)^\top \h{r} + \Pi b + \h{q}, \bar{x} \ran \\
& \hspace{0.3in} + 2 p^\top b + \sigma^\top P \sigma - \lan \h{R}(P) \theta^*, \theta^* \ran.
\end{aligned}
\end{equation*}
By the system of algebraic equations \eqref{eq:algebraic_Riccati} satisfied by $\{P, \Pi, p\}$, together with the expression for $c_0$ in \eqref{eq:c_0}, it follows that $(V(\cd, \cd), c_0)$ is a solution to the Bellman equation \eqref{eq:cell_problem}.
\end{proof}

\subsection{Verification theorem for the ergodic control problem}
\label{s:ergodic_control_problem}

We have observed that Problem (MFLQ)$^T$ can be regarded as a probabilistic interpretation of the Hamilton--Jacobi equation \eqref{eq:Hamilton-Jacobi_equation}. In this subsection, we aim to establish a similar connection for Problem (C) by demonstrating that Problem (EC) serves as a probabilistic interpretation of the Bellman equation in \eqref{eq:cell_problem}. Specifically, our objective is to delineate the connection between Problem (C) in \eqref{eq:cell_problem} and Problem (EC) in \eqref{eq:ergodic_control_problem},
which is referred to as the verification procedure in the context of control theory.

In a standard control problem, the objective is to find an optimal pair that minimizes the cost functional and thereby determines the value function. In contrast, Problem (EC) in \eqref{eq:ergodic_control_problem} inherently involves two objectives: identifying a pair $(\bar X(\cd),\bar u(\cd))$ that minimizes both the long-term average rate and the long-term residual cost simultaneously.

A key distinction in Problem (EC) is that the choice of $u(\cd)$ does not necessarily ensure the existence of $\lim_{T\to\i} J_T(x; u(\cd))$. In other words, the control space $\sU$ is not necessarily a subset of $\sU_{ad}[0,\i)$ (as defined in \eqref{eq:u_ad_infinity}). To address this, we introduce a refined control space $\sU$, which we define as a set
consisting of all $\dbF$-progressively measurable processes such that
\begin{itemize}
\item For all $x \in \dbR^n$, the state process $X(\cd) := X(\cd \,; x, u(\cd))$ of the solution to equation \eqref{eq:state} satisfies, for all $t \ges 0$,
\begin{equation*}
\dbE\big[|X(t)|^4 + |u(t)|^4 \big]< \i;
\end{equation*}

\item The law of $(X(t), u(t))$ converges to some distribution $\mu_\i \in \cP_2(\dbR^n \times \dbR^m)$ in the $2$-Wasserstein metric $\cW_2(\cd, \cd)$, i.e.,
$$\lim_{t\to \i} \cW_2(\cL(X(t), u(t)), \mu_\i) = 0.$$
\end{itemize}
We will construct a control $u(\cd) \in \sU$ in Theorem \ref{t:ergodic_control_problem} to demonstrate the non-emptiness of this set.

In the following, we provide a complete characterization of Problem (EC) in the LQ setting, utilizing the solvability result of Problem (C) in Lemma \ref{l:infinite_time_homo}. We first introduce the following SDE, whose solution, denoted by $\bar X(\cd)$, will subsequently be identified as the optimal state for Problem (EC)
\begin{equation}
\label{eq:optimal_path_ergodic}
\begin{cases}
\vspace{4pt}
\displaystyle
d \bar X(t) = \{(A + B \Th^*)(\bar X(t) - \dbE[\bar X(t)]) + (\h{A} + \h{B} \bar{\Th}^*) \dbE[\bar X(t)] + \h{B} \theta^* + b \} dt \\
\vspace{4pt}
\displaystyle
\hspace{0.6in} + \{(C + D \Th^*)(\bar X(t) - \dbE[\bar X(t)]) + (\h{C} + \h{D} \bar{\Th}^*) \dbE[\bar X(t)] + \h{D} \theta^* + \sigma \} dW(t), \\
\bar X(0) = x.
\end{cases}
\end{equation}
For simplicity of notation, we denote
\begin{equation*}
\begin{aligned}
& \wt{A}_1 = A + B \Th^*, \quad \wt{A}_2 = \bar A + \h{B} \bar \Th^* - B \Th^*, \quad \wt{b} = \h{B}\theta^* + b, \\
& \wt{C}_1 = C + D \Th^*, \quad \wt{C}_2 = \bar C + \h{D} \bar \Th^* - D \Th^*, \quad \wt{\sigma} = \h{D}\theta^* + \sigma.
\end{aligned}
\end{equation*}
Then, with these notations, the SDE \eqref{eq:optimal_path_ergodic} can equivalently be expressed as
\begin{equation*}
\begin{cases}
\vspace{4pt}
\displaystyle
d \bar X(t) = \big\{ \wt{A}_1 \bar X(t) + \wt{A}_2 \dbE[\bar X(t)] + \wt{b} \big\} dt + \big\{ \wt{C}_1 \bar X(t) + \wt{C}_2 \dbE[\bar X(t)] + \wt{\sigma} \big\} dW(t), \\
\bar X(0) = x.
\end{cases}
\end{equation*}

Next, we present two lemmas that play a key role in establishing the non-emptiness of $\sU$ and proving the verification theorem. In the following, we use $K$ and $\l$ to denote two generic positive constants which may vary from line to line. The first lemma demonstrates the boundedness of the moments of $\bar{X}(\cd)$.

\begin{lemma}
\label{l:moment_boundedness}
Let {\rm(H1)--(H2)} hold. Then, there exists a constant $K > 0$, independent of $t$, such that the solution $\bar X(\cd)$ to the SDE \eqref{eq:optimal_path_ergodic} satisfies, for all $t \ges 0$,
$$\dbE \big[|\bar X(t)|^2 \big] \les K, \quad \hbox{and} \quad \dbE \big[|\bar X(t)|^4 \big] < \infty.$$
\end{lemma}
\begin{proof}
First, note that
\begin{equation}
\label{eq:mean_ergodic}
d \dbE \big[\bar X(t) \big] = \big\{(\h{A} + \h{B} \bar \Th^*) \dbE \big[\bar X(t) \big] + \h{B} \theta^* + b \big\} dt, \quad \dbE \big[\bar X(0) \big] = x,
\end{equation}
which yields
$$\dbE \big[\bar X(t) \big] = e^{(\h{A} + \h{B} \bar \Th^*)t} x + \int_0^t e^{(\h{A} + \h{B} \bar \Th^*)(t-s)} (\h{B} \theta^* + b) ds.$$
From Lemma \ref{l:infinite_time_homo} (i), $\bar \Th^*$ is a stabilizer of \eqref{eq:homo_ode}, thus the matrix $\h{A} + \h{B} \bar \Th^*$ is stable, which implies that there exist some $\lambda > 0$ and $K > 0$ such that
$$\big\| e^{(\h{A} + \h{B} \bar \Th^*) t} \big\| \les K e^{-\lambda t}, \quad \forall t \ges 0.$$
Hence, we have the following estimate
\begin{equation*}
\begin{aligned}
\left\vert \dbE \left[\bar X(t) \right] \right\vert & \les K e^{-\lambda t} + K \int_0^t e^{-\lambda (t - s)} |\h{B} \theta^* + b| ds \\
& \les K e^{-\lambda t} + K \big(1- e^{-\lambda t} \big) \les K,
\end{aligned}
\end{equation*}
which gives that, for any $l > 0$,
\begin{equation}
\label{eq:boundedness_expectation}
\left\vert \dbE \left[\bar X(t) \right] \right\vert^l \les K, \quad \forall t \ges 0.
\end{equation}
Then, by using (D.5) of \cite{Fleming-Soner-2006}, for any $l \ges 2$, we derive the following result
\begin{equation*}
\begin{aligned}
\dbE\Big[\sup_{t \in [0, T]} |\bar X(t)|^l \Big] &\les K|x|^l + KT^{\frac{l}{2} - 1} e^{KT} \dbE \Big[\int_0^T \Big\{|x|^l + \left(|\wt{A}_2| + |\wt{b}|\right)^l + \left(|\wt{C}_2| + |\wt{\sigma}|\right)^l \Big\} dt \Big] < \infty.
\end{aligned}
\end{equation*}

Next, from Lemma \ref{l:infinite_time_homo} (i), $\Th^*$ is a stabilizer of \eqref{eq:homo_sde}. Consequently, the homogeneous system $[A + B \Th^*, C+ D\Th^*]$ is $L^2$-exponentially stable. Thus, Lemma 2.2 (iv) in \cite{Jian-Jin-Song-Yong-2024} guarantees that the following Lyapunov equation admits a unique solution $\bar{P} \in \dbS^n_{++}$:
\begin{equation}
\label{eq:Lyapunov_equation}
\begin{aligned}
& \bar{P}(A + B \Th^*) + (A + B \Th^*)^\top \bar{P} + (C + D \Th^*)^\top \bar{P} (C + D \Th^*) + I_n \\
= \ & \bar{P} \wt{A}_1 + \wt{A}_1^\top \bar{P} + \wt{C}_1^\top \bar{P} \wt{C}_1 + I_n = 0.
\end{aligned}
\end{equation}
Let $\beta > 0$ be the largest eigenvalue of $\bar{P}$. Applying It\^o's formula, we obtain
\begin{equation*}
\begin{aligned}
& \frac{d}{dt} \dbE\left[\lan \bar{P} \bar X(t), \bar X(t) \right] \\
= \ & \dbE\Big[\big\lan \big(\bar{P}\wt{A}_1 + \wt{A}_1^\top \bar{P} + \wt{C}_1^\top \bar{P} \wt{C}_1 \big) \bar X(t), \bar X(t) \big\ran \\
& \hspace{0.5in} + \big\lan \big( \bar{P} \wt{A}_2 + \wt{A}_2^\top \bar{P} + \wt{C}_2^\top \bar{P} \wt{C}_2 + \wt{C}_2^\top \bar{P} \wt{C}_1 + \wt{C}_1^\top \bar{P} \wt{C}_2 \big) \dbE[\bar X(t)], \dbE[\bar X(t)] \big\ran \\
& \hspace{0.5in} + \big\lan 2\bar{P}\wt{b} + 2\wt{C}_1^\top \bar{P} \wt{\sigma} + 2 \wt{C}_2^\top \bar{P} \wt{\sigma}, \dbE[\bar X(t)] \big\ran + \lan \bar{P}\wt{\sigma}, \wt{\sigma} \ran \Big] \\
\les \ & - \dbE\left[|\bar X(t)|^2 \right] + K \\
\les \ & - \frac{1}{\beta} \dbE\left[\lan \bar{P}\bar X(t), \bar X(t) \ran\right] + K 
\end{aligned}
\end{equation*}
by using the result in \eqref{eq:boundedness_expectation}. From Gr\"onwall's inequality, we get the following estimate
\begin{equation*}
\begin{aligned}
\dbE\left[\lan \bar{P}\bar X(t), \bar X(t) \right] \les K e^{- \frac{1}{\beta} t} |x|^2 + K \int_0^t e^{- \frac{1}{\beta} (t-s)} ds \les K
\end{aligned}
\end{equation*}
for all $t \ges 0$, where $K$ is a positive constant independent of $t$. By the positivity of the matrix $\bar{P}$, we conclude the desired result that
$$\dbE\left[|\bar X(t)|^2 \right] \les K, \quad \forall t \ges 0.$$
\end{proof}

The second lemma characterizes the asymptotic behavior of the optimal path $\bar{X}(\cd)$ in \eqref{eq:optimal_path_ergodic} for Problem (EC).

\begin{lemma}
\label{l:invariant-distribution}
Let {\rm(H1)--(H2)} hold. Then, there exists a unique measure $\nu_{\i} \in \cP_2(\dbR^n)$ such that
$$\cW_2(\cL(\bar X(t), \nu_{\i}) \to 0, \quad \hbox{as } t \to \infty.$$
\end{lemma}
\begin{proof} 
Let $\nu_t = \cL(\bar X(t))$. We claim that it is Cauchy in $\cP_2(\dbR^n)$. Let $\Phi(\cd)$ denote the fundamental matrix of the homogeneous system $[\wt{A}_1, \wt{C}_1] = [A + B \Th^*, C+ D\Th^*]$, i.e.,
\begin{equation*}
d \Phi(t) = \wt{A}_1 \Phi(t) dt + \wt{C}_1 \Phi(t) dW(t), \quad \Phi(0) = I_n.
\end{equation*}
It is known that $\Phi(t)$ is invertible for all $t \ges 0$, and the following equation holds
\begin{equation*}
d \Phi(t)^{-1} = \Phi(t)^{-1}(\wt{C}_1^2 - \wt{A}_1) dt - \Phi(t)^{-1} \wt{C}_1 dW(t), \quad \Phi(0)^{-1} = I_n.
\end{equation*}
Applying the product rule, we then obtain
\begin{equation*}
\begin{aligned}
d(\Phi(t)^{-1} \bar X(t)) &= \Phi(t)^{-1} \big(\wt{A}_2 \dbE[\bar X(t)] - \wt{C}_1 \wt{C}_2 \dbE[\bar X(t)] + \wt{b} - \wt{C}_1 \wt{\sigma} \big) dt \\
& \hspace{0.5in} + \Phi(t)^{-1} \big(\wt{C}_2 \dbE[\bar X(t)] + \wt{\sigma} \big) dW(t),
\end{aligned}
\end{equation*}
which yields
\begin{equation*}
\begin{aligned}
\bar X(t) &= \Phi(t) x + \int_0^t \Phi(t) \Phi(r)^{-1} \big((\wt{A}_2 - \wt{C}_1 \wt{C}_2) \dbE[\bar X(r)] + \wt{b} - \wt{C}_1 \wt{\sigma} \big) dr \\
& \hspace{0.5in} + \int_0^t \Phi(t) \Phi(r)^{-1} \big(\wt{C}_2 \dbE[\bar X(r)] + \wt{\sigma} \big) dW(r).
\end{aligned}
\end{equation*}
Let $\Phi(s, t) = \Phi(t) \Phi(s)^{-1}$. By a similar argument, for $s, t \ges 0$, the process $\bar X(\cd)$ satisfies the following representation
\begin{equation*}
\begin{aligned}
\bar X(t+s) &= \Phi(s, s+t) \bar X(s) + \Phi(s, s+t)  \int_s^{s+t} \Phi(s, r)^{-1} \big((\wt{A}_2 - \wt{C}_1 \wt{C}_2) \dbE[\bar X(r)] + \wt{b} - \wt{C}_1 \wt{\sigma} \big) dr \\
& \hspace{0.5in} + \Phi(s, s+t) \int_s^{s+t} \Phi(s, r)^{-1} \big(\wt{C}_2 \dbE[\bar X(r)] + \wt{\sigma} \big) dW(r).
\end{aligned}
\end{equation*}

Next, let $\wt{\Phi}(\cd)$ be the solution to the following SDE
\begin{equation*}
d \wt{\Phi}(t) = \wt{A}_1 \wt{\Phi}(t) dt + \wt{C}_1 \wt{\Phi}(t) d\wt{W}(t), \quad \wt{\Phi}(0) = I_n,
\end{equation*}
where $\wt{W}(\cd)$ is another standard Brownian motion, independent of $W(\cd)$. According to the definition of $L^2$-exponential stability in Lemma 2.2 (i) of \cite{Jian-Jin-Song-Yong-2024}, there exist $K > 0$ and $\lambda > 0$ such that
\begin{equation*}
\dbE \big[\|\wt{\Phi}(t)\|^2 \big] \les K e^{-\lambda t}, \quad \forall t \ges 0.
\end{equation*}
For each $s \ges 0$, let $\xi_s$ be a random variable such that $\cL(\xi_s) = \cL(\bar{X}(s))$, and assume that $\xi_s$ is independent of $\wt{W}(\cd)$. We define $\wt{X}^s(\cd)$ as the solution to the following equation
\begin{equation*}
\begin{aligned}
\wt{X}^s(t) &= \wt{\Phi}(t) \xi_s + \wt{\Phi}(t)  \int_0^{t} \wt{\Phi}(r)^{-1} \big((\wt{A}_2 - \wt{C}_1 \wt{C}_2) \dbE[\bar{X}(s+r)] + \wt{b} - \wt{C}_1 \wt{\sigma} \big) dr \\
& \hspace{0.5in} + \wt{\Phi}(t)  \int_0^{t} \wt{\Phi}(r)^{-1} \big(\wt{C}_2 \dbE[\bar{X}(s+r)] + \wt{\sigma} \big) d \wt{W}(r).
\end{aligned}
\end{equation*}
Indeed, $\wt{X}^s(\cd)$ is equivalently the solution of
\begin{equation*}
\begin{cases}
\vspace{4pt}
\displaystyle
d \wt{X}^s(t) = \big\{ \wt{A}_1 \wt{X}^s(t) + \wt{A}_2 \dbE[\bar X(s + t)] + \wt{b} \big\} dt + \big\{ \wt{C}_1 \wt{X}^s(t) + \wt{C}_2 \dbE[\bar X(s+t)] + \wt{\sigma} \big\} d\wt{W}(t), \\
\wt{X}^s(0) = \xi_s.
\end{cases}
\end{equation*}
Thus, $\wt{X}^s(t)$ and $\bar X(s+t)$ have the same distribution. Specifically, let $\wt{X}^0(0) =x$, then $\wt{X}^0(t)$ provides a distributional copy to $\bar X(t)$. 

By \eqref{eq:mean_ergodic}, since $\h{A} + \h{B} \bar{\Theta}^*$ is stable, the equilibrium mean is well-defined and it is given by
$$\bar{x}_{\infty} = - \big(\h{A} + \h{B} \bar{\Theta}^* \big)^{-1} (\h{B} \theta^* + b).$$
Moreover, we have
$$\dbE[\bar{X}(t)] - \bar{x}_{\infty} = e^{(\h{A} + \h{B} \bar{\Theta}^*)t}(x - \bar{x}_{\infty}),$$
and thus there exist constants $K$ and $\lambda > 0$ such that
$|\dbE[\bar{X}(t)] - \bar{x}_{\infty}|^2 \les K e^{-\lambda t}$ for all $t \ges 0$.
Consequently, for all $s, t \ges 0$, 
$$\big|\dbE[\bar{X}(s+t)] - \dbE[\bar{X}(t)] \big|^2 \les 2 \big|\dbE[\bar{X}(s+t)] - \bar{x}_{\infty} \big|^2 + 2 \big| \dbE[\bar{X}(t)] - \bar{x}_{\infty} \big|^2 \les 4K e^{-\lambda t}.$$

We define $\Delta^s(t) := \wt{X}^s(t) - \wt{X}^0(t)$ for $t \ges 0$. It is clear that $\Delta^s(\cd)$ satisfies the following SDE
\begin{equation*}
\begin{cases}
\vspace{4pt}
\displaystyle
d \Delta^s(t) = \big\{ \wt{A}_1 \Delta^s(t) + \wt{A}_2 \big(\dbE[\bar X(s + t)] - \dbE[\bar X(t)] \big) \big\} dt \\
\hspace{0.7in} + \big\{ \wt{C}_1 \Delta^s(t) + \wt{C}_2 \big(\dbE[\bar X(s+t)] - \dbE[\bar X(t)] \big) \big\} d\wt{W}(t), \\
\Delta^s(0) = \xi_s - x.
\end{cases}
\end{equation*}
Let $\bar{P}$ be the solution to the Lyapunov equation \eqref{eq:Lyapunov_equation}, and let $\beta > 0$ be the largest eigenvalue of $\bar{P}$. By It\^o's formula and Young's inequality, applying a similar argument as Lemma \ref{l:moment_boundedness}, we obtain
\begin{equation*}
\begin{aligned}
\frac{d}{dt} \dbE\left[\lan \bar{P} \Delta^s(t), \Delta^s(t) \right] 
& \les - \frac{1}{2} \dbE\left[|\Delta^s(t)|^2 \right] + K \big|\dbE[\bar X(s + t)] - \dbE[\bar X(t)] \big|^2 \\
& \les - \frac{1}{2 \beta} \dbE\left[\lan \bar{P} \Delta^s(t), \Delta^s(t) \right] + K e^{-\lambda t} 
\end{aligned}
\end{equation*}
for some $K > 0$. By Lemma \ref{l:moment_boundedness}, there exists $K > 0$ such that $\dbE[|\bar{X}(s)|^2] = \dbE[|\xi_s|^2] \les K$ for all $s \ges 0$. It follows that
\begin{equation*}
\dbE\left[\lan \bar{P} \Delta^s(t), \Delta^s(t) \right] \les K e^{-\frac{1}{2 \beta} t} \dbE \big[|\xi_s - x|^2 \big] + K \int_0^t e^{-\frac{1}{2\beta}(t-r)} e^{-\lambda r} dr \les K e^{-\lambda t}, \quad \forall t \ges 0,
\end{equation*}
for some possible different $K$ and $\lambda$. Since $\bar{P}$ is positive definite, we conclude that for all $s, t \ges 0$,
$$\dbE \big[|\Delta^s(t)|^2 \big] = \dbE \big[|\wt{X}^s(t) - \wt{X}^0(t)|^2 \big] \les K e^{-\lambda t}.$$
Hence, for all $s, t \ges 0$,
\begin{equation*}
\begin{aligned}
\cW_2^2(\nu_{s+t}, \nu_t) = \cW_2^2 \big(\cL(\wt{X}^s(t)), \cL(\wt{X}^0(t)) \big) \les \dbE \big[|\wt{X}^s(t) - \wt{X}^0(t)|^2 \big] \les K e^{-\lambda t},
\end{aligned}
\end{equation*}
where $K$ and $\lambda$ are two positive constants independent of $s$ and $t$. This implies that $\{\nu_t\}_{t \ges 0}$ is a Cauchy sequence in $\cP_2(\dbR^n)$, and the desired conclusion follows from the completeness of the space $(\cP_2(\dbR^n), \cW_2)$.
\end{proof}

Now, we are ready to present the main result in this subsection. The following theorem provides a comprehensive characterization of Problem (EC) within the context of the LQ framework, leveraging the results of Problem (C) established in Lemma \ref{l:infinite_time_homo}. To proceed, for each $a \in \dbR^n$, we define the mapping $\pi_a: \dbR^n \times \dbR^m \to \dbR^n$ by $\pi_a(x, u) = x - a$. For any $\mu \in \cP_2(\dbR^n \times \dbR^m)$, we denote by $(\pi_a)_{\#}\mu$ the push-forward measure of $\mu$ under $\pi_a$.

\begin{theorem}
\label{t:ergodic_control_problem}
Suppose {\rm(H1)--(H2)} hold. Let $(V(\cd, \cd), c_0)$ be the solution to Problem {\rm(C)}, with the explicit form
\begin{equation*}
V(\bar{x}, \nu) = \int_{\dbR^n} \lan P y, y \ran \nu(dy) + \lan \Pi \bar{x}, \bar{x} \ran + 2 \lan p, \bar{x} \ran,
\end{equation*}
where $\{P, \Pi, p\}$ is the unique solution to the system of algebraic equations \eqref{eq:algebraic_Riccati}, and $c_0$ is determined by \eqref{eq:c_0}.
Let $\bar X(\cd)$ be the solution to the SDE \eqref{eq:optimal_path_ergodic},
and $\bar u(\cd)$ be expressed as
\begin{equation}
\begin{aligned}
\label{eq:optimal_control_EC}
\bar u(t) &:= \bar u_{\infty} \left(\bar X(t) - \dbE[\bar X(t)], \dbE[\bar X(t)]\right) = \Th^{*} (\bar X(t) - \dbE[\bar X(t)] ) + \bar \Th^{*} \dbE[\bar X(t)]  + \theta^{*},
\end{aligned}
\end{equation}
where $\bar u_{\infty}: \dbR^n \times \dbR^n \mapsto \dbR^m$ is given by \eqref{eq:feedback_form}, and the coefficients $(\Th^*, \bar\Th^*, \theta^*)$ are defined by \eqref{eq:theta_star}. Then, the following results hold:

{\rm(i)} The law of the pair $(\bar X(t), \bar u(t))$ converges to some $\bar\mu^{*}_\infty \in \cP_2(\dbR^n \times \dbR^m)$ in the 2-Wasserstein distance, as $t\to \i$;

{\rm(ii)} The constant $c_0$, referred to as the ergodic constant, is uniquely determined by the following limit
\begin{equation}
\label{eq:ergodic_constant}
c_0=\lim_{T\to\i} \frac{1}{T} \int_0^T \dbE \big[f(\bar X(t), \dbE[\bar X(t)], \bar u(t), \dbE[\bar{u}(t)]) \big] dt;
\end{equation}

{\rm(iii)} Let $\bar{x}^{*}_{\infty} := \int_{\dbR^n \times \dbR^m} x \bar{\mu}^{*}_{\infty}(dx, du)$ and $\nu^{*}_{\infty} := (\pi_{\bar{x}^{*}_{\infty}})_{\#}\bar{\mu}^{*}_{\infty}$. Then, the 4-tuple $\{U(\cd), c_0, \bar X(\cd), \bar u(\cd)\}$ solves Problem {\rm (EC)}, where 
$$U(x) := V(x, \delta_0) - V(\bar{x}^{*}_\infty, \nu^{*}_{\infty}), \quad \forall x \in \dbR^n.$$
\end{theorem}

\begin{proof}
By \eqref{eq:optimal_control_EC}, $\bar{u}(t)$ is affine with respect to $\bar{X}(t)$ and $\dbE[\bar{X}(t)]$, thus the result (i) follows directly from Lemma \ref{l:invariant-distribution}. For any control $u(\cd) \in \sU$ and its associated state process $X(\cd):= X(\cd \, ; x, u(\cd))$, we define $Y(t) = X(t) - \dbE[X(t)]$, $Z(t) = u(t) - \dbE[u(t)]$, and
\begin{equation*}
\nu_t = \cL(Y(t)), \quad \mu_t = \cL(Y(t), Z(t)), \quad \bar{x}_t = \dbE[X(t)], \quad \bar{u}_t = \dbE[u(t)]
\end{equation*}
for all $t \ges 0$. Then, $\mu_t \in \cA(\nu_t)$ because its first marginal is $\nu_t$ and $\int_{\dbR^n \times \dbR^m} z \mu_t(dy, dz) = \dbE[Z(t)] = 0$. By the definition of $\sU$, there exists a measure $\mu_{\infty}^u \in \cP_2(\dbR^n \times \dbR^m)$ such that
$$\lim_{t \to \infty} \cW_2(\cL(X(t), u(t)), \mu^u_{\infty}) = 0.$$
We set
$$\bar{x}_{\infty}^u := \int_{\dbR^n \times \dbR^m} x \mu^u_{\infty}(dx, du), \quad \bar{u}^u_{\infty} := \int_{\dbR^n \times \dbR^m} u \mu^u_{\infty} (dx, du),$$
and define the limiting centered law 
$\nu_{\infty}^u := (\pi_{\bar{x}^u_{\infty}})_{\#} \mu_{\infty}^u$. Then, $\lim_{t \to \infty} \bar{x}_t = \bar{x}^{u}_{\infty}$, $\lim_{t \to \infty} \bar{u}_t = \bar{u}^{u}_{\infty}$, and $\lim_{t \to \infty} \cW_2(\nu_t, \nu^{u}_{\infty}) = 0$. 
Let $(V(\cd, \cd), c_0)$ be the solution to Problem (C), given by \eqref{eq:value_function_cell} and \eqref{eq:c_0}, respectively. Applying It\^o's formula along a flow of measures (for example, see Theorem 5.92 in \cite{Carmona-Delarue-2018}), we obtain
\begin{equation*}
\begin{aligned}
V(\bar{x}_t, \nu_t) &= V(\bar{x}_0, \nu_0) + \int_0^t \int_{\dbR^n \times \dbR^m} \Big\{ \lan Ay + Bz, D_{\nu} V(\bar{x}_s, \nu_s)(y) \ran \\
&\hspace{0.3in} + \frac{1}{2} \lan D_yD_{\nu} V(\bar{x}_s, \nu_s)(y) \h{\sigma}(y, \bar{x}_s, z, \bar{u}_s), \h{\sigma}(y, \bar{x}_s, z, \bar{u}_s) \ran \Big\} \mu_s(dy, dz) ds \\
& \hspace{0.3in} + \int_0^t \lan \h{A} \bar{x}_s + \h{B} \bar{u}_s + b, D_{\bar{x}} V(\bar{x}_s, \nu_s) \ran ds.
\end{aligned}
\end{equation*}
By the definition of $\cH_{\infty}$ in \eqref{eq:Hamiltonian_infty}, the above equation can be equivalently rewritten as
\begin{equation*}
\begin{aligned}
V(\bar{x}_t, \nu_t) &= V(\bar{x}_0, \nu_0) + \int_0^t \Big(\cH_{\infty}(\bar{x}_s, \nu_s, \bar{u}_s, \mu_s; V) - \int_{\dbR^n \times \dbR^m} \wt{f}(y, \bar{x}_s, z, \bar{u}_s) \mu_s(dy, dz) \Big) ds,
\end{aligned}
\end{equation*}
where $\wt{f}$ is defined in \eqref{eq:f_tilde}. Observe that
$$\int_{\dbR^n \times \dbR^m} \wt{f}(y, \bar{x}_s, z, \bar{u}_s) \mu_s(dy, dz) = \dbE[f(X(s), \dbE[X(s)], u(s), \dbE[u(s)])].$$
Combining these two identities, we derive that
\begin{equation}
\label{eq:verification_identity}
\begin{aligned}
& \dbE \Big[\int_0^t \big( f(X(s), \dbE[X(s)], u(s), \dbE[u(s)]) - c_0 \big) ds\Big] \\
= \ & V(\bar{x}_0, \nu_0) - V(\bar{x}_t, \nu_t) + \int_0^t \big(\cH_{\infty}(\bar{x}_s, \nu_s, \bar{u}_s, \mu_s; V) - c_0 \big) ds.
\end{aligned}
\end{equation}
By the Bellman equation \eqref{eq:cell_problem}, for all $\bar{u}_s \in \dbR^m$ and $\mu_s \in \cA(\nu_s)$, we have
\begin{equation*}
c_0 \les \cH_{\infty}(\bar{x}_s, \nu_s, \bar{u}_s, \mu_s; V), \quad \forall \bar{x}_s \in \dbR^n, \, \nu_s \in \cP_2^0(\dbR^n),
\end{equation*}
which implies that
\begin{equation}
\label{eq:verification_inequality}
V(\bar{x}_0, \nu_0) \les V(\bar{x}_t, \nu_t) + \dbE \Big[\int_0^t \big( f(X(s), \dbE[X(s)], u(s), \dbE[u(s)]) - c_0 \big) ds\Big].
\end{equation}
The inequality \eqref{eq:verification_inequality} holds for all $u(\cd) \in \sU$. It is worth noting that Lemma \ref{l:moment_boundedness} and \ref{l:invariant-distribution}, together with equation \eqref{eq:optimal_control_EC} imply that $\bar u(\cd)\in\sU$. This, in turn, confirms that $\sU \neq \varnothing$. Moreover, setting $u(\cd) = \bar u(\cd)$, the above inequality becomes an equality. Next, since the value function $V(\cd, \cd)$ is continuous and satisfies a quadratic growth condition, we have $\lim_{t \to \infty} V(\bar{x}_t, \nu_t) = V(\bar{x}^u_{\infty}, \nu_{\infty}^u)$. Similarly, by the definition of $f(\cd, \cd, \cd, \cd)$ in \eqref{eq:running_cost}, it is also continuous and has quadratic growth, thus 
\begin{equation*}
\lim_{t \to \infty} \dbE[f(X(t), \dbE[X(t)], u(t), \dbE[u(t)])] = \int_{\dbR^n \times \dbR^m} f(x, \bar{x}^u_{\infty}, u, \bar{u}^u_{\infty}) \mu^{u}_{\infty}(dx, du).
\end{equation*}
Therefore, dividing both sides of \eqref{eq:verification_inequality} by $t$ and letting $t \to \infty$, we conclude that
\begin{equation}
\label{eq:inequality_c0}
c_0 \les \lim_{t \to \infty} \frac{1}{t} \int_0^t \dbE[f(X(s), \dbE[X(s)], u(s), \dbE[u(s)])] ds = \int_{\dbR^n \times \dbR^m} f(x, \bar{x}^u_{\infty}, u, \bar{u}^u_{\infty}) \mu^{u}_{\infty}(dx, du).
\end{equation}
The above inequality holds for all $u(\cd) \in \sU$. Moreover, equality is attained by choosing $u(\cdot)=\bar u(\cdot)$, for which $\mu_{\i}^{\bar u} = \bar{\mu}^{*}_{\i}$. Consequently, we establish the desired result
$$c_0 = \lim_{T \to \infty} \frac{1}{T} \int_0^T \dbE[f(\bar X(t), \dbE[\bar X(t)], \bar u(t), \dbE[\bar{u}(t)])] dt = \int_{\dbR^n \times \dbR^m} f(x, \bar{x}^*_{\infty}, u, \bar{u}^*_{\infty}) \bar{\mu}^{*}_{\infty}(dx, du)$$
and thus conclude part (ii) of Theorem \ref{t:ergodic_control_problem}, where
\begin{equation*}
\bar{x}^{*}_{\infty} = \int_{\dbR^n \times \dbR^m} x \bar{\mu}^{*}_{\infty}(dx, du), \quad \text{and} \quad \bar{u}^{*}_{\infty} = \int_{\dbR^n \times \dbR^m} u \bar{\mu}^{*}_{\infty} (dx, du).
\end{equation*}

Next, as $t$ goes to infinity, by \eqref{eq:verification_inequality}, we deduce the following inequality
$$V(\bar{x}_0, \nu_0) \les V(\bar{x}^u_\infty, \nu^u_{\infty}) + \lim_{t \to \infty} \dbE \Big[\int_0^t \big( f(X(s), \dbE[X(s)], u(s), \dbE[u(s)]) - c_0 \big) ds\Big],$$
which yields that, for all $u(\cd) \in \sU$,
\begin{equation}
\label{eq:value_inequality}
V(\bar{x}_0, \nu_0) \les V(\bar{x}^{*}_\infty, \nu^{*}_{\infty}) + \lim_{t \to \infty} \dbE \Big[\int_0^t \big( f(X(s), \dbE[X(s)], u(s), \dbE[u(s)]) - c_0 \big) ds\Big],
\end{equation}
where $\nu^{*}_{\infty} = (\pi_{\bar{x}^{*}_{\infty}})_{\#}\bar{\mu}^{*}_{\infty}$. This is because
\begin{enumerate}
\item[(i)] Firstly, using \eqref{eq:ergodic_constant} and \eqref{eq:inequality_c0}, for all $u(\cd) \in \sU$, we have
$$\int_{\dbR^n \times \dbR^m} f(x, \bar{x}^u_{\infty}, u, \bar{u}^u_{\infty}) \mu^{u}_{\infty}(dx, du) \ges \int_{\dbR^n \times \dbR^m} f(x, \bar{x}^*_{\infty}, u, \bar{u}^*_{\infty}) \bar{\mu}^{*}_{\infty}(dx, du) = c_0.$$

\item[(ii)] Then, the inequality \eqref{eq:value_inequality} can be verified by examining two cases:
\begin{itemize}
\item If $\int_{\dbR^n \times \dbR^m} f(x, \bar{x}^u_{\infty}, u, \bar{u}^u_{\infty}) \mu^{u}_{\infty}(dx, du) = \int_{\dbR^n \times \dbR^m} f(x, \bar{x}^*_{\infty}, u, \bar{u}^*_{\infty}) \bar{\mu}^{*}_{\infty}(dx, du) = c_0$, by the identity \eqref{eq:verification_identity}, we obtain
\begin{equation*}
\begin{aligned}
0  & = \lim_{t \to \infty} \frac{1}{t} \dbE \Big[\int_0^t \big( f(X(s), \dbE[X(s)], u(s), \dbE[u(s)]) - c_0 \big) ds\Big] \\
& = \lim_{t \to \infty} \frac{1}{t} \int_0^t \big(\cH_{\infty}(\bar{x}_s, \nu_s, \bar{u}_s, \mu_s; V) - c_0 \big) ds,
\end{aligned}
\end{equation*}
which implies that $\mu^u_{\infty} = \bar{\mu}^{*}_{\infty}$ by the Bellman equation \eqref{eq:cell_problem} and the uniform convexity of $\cH_{\infty}$. Thus, \eqref{eq:value_inequality} holds with $\mu_{\i}^u$ replaced by $\bar{\mu}^{*}_{\i}$;
\item If $\int_{\dbR^n \times \dbR^m} f(x, \bar{x}^u_{\infty}, u, \bar{u}^u_{\infty}) \mu^{u}_{\infty}(dx, du) > \int_{\dbR^n \times \dbR^m} f(x, \bar{x}^*_{\infty}, u, \bar{u}^*_{\infty}) \bar{\mu}^{*}_{\infty}(dx, du)$, then
$$\lim_{t \to \i} \int_0^{t} \big( \dbE[ f(X(s), \dbE[X(s)], u(s), \dbE[u(s)])] - c_0 \big) ds = \i,$$
which makes the right-hand side of \eqref{eq:value_inequality} infinite.
\end{itemize}
\end{enumerate}
Furthermore, \eqref{eq:value_inequality} holds as an equality when $u(\cd) = \bar u(\cd)$. Thus, we can rewrite it as
$$V(\bar{x}_0, \nu_0) = V(\bar{x}^{*}_\infty, \nu^{*}_{\infty}) + \inf_{u(\cd) \in \sU} \lim_{t \to \infty} \dbE \Big[\int_0^t \big( f(X(s), \dbE[X(s)], u(s), \dbE[u(s)]) - c_0 \big) ds\Big].$$
This implies that
$$U(x) = V(\bar{x}_0, \nu_0) - V(\bar{x}^{*}_\infty, \nu^{*}_{\infty}) = V(x, \delta_0) - V(\bar{x}^{*}_\infty, \nu^{*}_{\infty}),$$
where the last equality follows from $\bar{x}_0 = x$ and $\nu_0 = \delta_0$. This establishes the result (iii) of Theorem \ref{t:ergodic_control_problem}. 
\end{proof}


\section{Convergence and turnpike properties}
\label{s:convergence_and_turnpike}

In this section, we first derive some key estimates that connect the coefficient functions $P_T(\cdot)$, $\Pi_T(\cd)$, $p_T(\cdot)$,
$\Theta_T^*(\cdot)$, $\bar \Th_T^*(\cd)$, and $\theta_T^*(\cdot)$, which arise in Problem (MFLQ)$^T$ from Section \ref{s:Finite_MFLQ}, to the corresponding constant matrices and vectors $P$, $\Pi$, $p$, $\Theta^*$, $\bar \Th^*$ and $\theta^*$ from Problem (C) in Subsection \ref{s:cell_problem}. Then, leveraging these natural convergence results, we demonstrate the stochastic turnpike property between the optimal pairs of Problem (MFLQ)$^T$ and Problem (EC). Finally, we establish a connection between Problem (C) and the ergodic cost in \eqref{eq:ergodic_cost}.

\subsection{Convergence of coefficients}
\label{s:Estimates}

First, under the assumptions (H1)--(H2), Lemma 2.3 in \cite{Sun-Yong-2024} proved that
\begin{equation}
\label{eq:estimates_P}
\|P_T(t)-P\|\les Ke^{-\l(T-t)}, \quad \forall t \in [0,T],
\end{equation}
for some $K,\l>0$, independent of $T$, where $P_T(\cd)$ is the positive semi-definite solution of Riccati differential equation for $P_T(\cd)$ in \eqref{eq:Riccati_ODE}
and $P$ is the stabilizing solution to algebraic Riccati equation \eqref{eq:ARE_P}, respectively. From \eqref{eq:estimates_P}, it is clear that
\begin{equation*}
\|P_T(t)\|\les K, \quad \forall t \in [0,T].
\end{equation*}

Next, by Theorem 4.1 in \cite{Sun-Yong-2024}, under the same assumptions (H1)--(H2), there exist some positive constants $K$ and $\l$, independent of $T$, such that
\begin{equation}
\label{eq:estimates_Pi}
\|\Pi_T(t)-\Pi\| \les Ke^{-\l(T-t)}, \quad \forall t \in [0,T],
\end{equation}
where $\Pi_T(\cd)$ is the positive semi-definite solution of Riccati differential equation for $\Pi_T(\cd)$ in \eqref{eq:Riccati_ODE}
and $\Pi$ is the positive definite solution of algebraic Riccati equation
for $\Pi$ in \eqref{eq:algebraic_Riccati}. 

\ms

The following lemma provides the remaining quantitative estimates that will be used to establish the desired turnpike property.

\begin{lemma}
\label{l:estimates}
Let {\rm(H1)--(H2)} hold. Then, there exist some constants $K, \l > 0$, independent of $T$, such that
\begin{equation}
\label{eq:estimates_Theta_T}
\|\Th_T^*(t)-\Th^*\| \les K e^{-\l(T-t)}, \quad \forall t \in [0, T],
\end{equation}
\begin{equation}
\label{eq:estimates_bar_Theta_T}
\|\bar \Th_T^*(t) - \bar \Th^*\| \les K e^{-\l(T-t)}, \quad \forall t \in [0, T],
\end{equation}
\begin{equation}
\label{eq:estimates_p}
|p_T(t)-p| \les K e^{-\l(T-t)}, \quad \forall t \in [0, T],
\end{equation}
\begin{equation}
\label{eq:estimates_theta}
|\th_T^*(t)-\th^*| \les K e^{-\l(T-t)}, \quad \forall t \in [0, T].
\end{equation}
Consequently, $\Th_T^*(\cd)$, $\bar \Th_T^*(\cd)$, $\th_T^*(\cd)$, and $p_T(\cd)$ are uniformly bounded on $[0, T]$.
\end{lemma}

\begin{proof}
Firstly, by definition and using the result in \eqref{eq:estimates_P}, we have
\begin{equation*}
\begin{aligned}
\|\Th_T^*(t)-\Th^*\| & = \big\|- \cR(P_T(t))^{-1} \cS(P_T(t)) + \cR(P)^{-1} \cS(P) \big\| \\
& \les \big\| \big([R+D^\top P_T(t) D]^{-1}-[R+D^\top PD]^{-1} \big) [B^\top P_T(t)+D^\top P_T(t)C+S] \big\| \\
&\hspace{0.5in} + \big\|[R+D^\top PD]^{-1} \big([B^\top P_T(t)+D^\top P_T(t)C] - [B^\top P+D^\top PC] \big) \big\|\\
& \les Ke^{-\l(T-t)}
\end{aligned}
\end{equation*}
for all $t \in [0,T]$. This result follows from the observation that 
$$[R+D^\top P_T(t) D]^{-1}-[R+D^\top PD]^{-1} = [R+D^\top P_T(t) D]^{-1} D^\top (P - P_T(t)) D [R+D^\top PD]^{-1}.$$
Thus, \eqref{eq:estimates_Theta_T} is established, which in turn implies that $\Th_T^*(t)$ is uniformly bounded on $[0, T]$. Then, by a similar argument and utilizing the results in \eqref{eq:estimates_P} and \eqref{eq:estimates_Pi}, we further obtain
$$\|\bar \Th_T^*(t) - \bar \Th^*\| = \|- \h{\cR}(P_T(t))^{-1} \h{\cS}(P_T(t), \Pi_T(t)) + \h{\cR}(P)^{-1} \h \cS(P, \Pi)\| \les K e^{-\l(T-t)}$$
for all $t \in [0,T]$, which establishes that $\bar \Th_T^*(t)$ is also uniformly bounded on $[0, T]$. 

Next, we aim to show \eqref{eq:estimates_p}. Recall that $p_T(\cd) \in C^1([0, T]; \dbR^n)$ satisfies
\begin{equation*}
\begin{cases}
\vspace{4pt}
\displaystyle
\dot{p}_T(t) + (\h{A} + \h{B}\bar{\Th}_T^*(t))^\top p_T(t) + (\h{C} + \h{D} \bar{\Th}_T^*(t))^\top P_T(t) \sigma \\
\vspace{4pt}
\displaystyle
\hspace{0.3in} + \bar{\Th}_T^*(t)^\top  \h{r} + \Pi_T(t) b + \h{q} = 0, \\
p_T(T) = 0,
\end{cases}
\end{equation*}
and $p$ is the solution to the equation
$$(\h{A} + \h{B}\bar{\Th}^*)^\top p + (\h{C} + \h{D} \bar{\Th}^*)^\top P \sigma + (\bar{\Th}^*)^\top \h{r} + \Pi b + \h{q} = 0.$$
The uniform boundedness of $p_T(\cd)$ on $[0, T]$ can be obtained by a similar argument from Theorem 4.1 in \cite{Jian-Jin-Song-Yong-2024}. To compare $p_T(\cd)$ and $p$, we denote the difference as
$$\h p(t) = p_T(t) - p, \quad \forall t \in [0, T],$$
and it is clear that
\begin{equation*}
\begin{aligned}
- \frac{d}{dt} \h p(t) &= [\h{A} + \h{B}\bar \Th^*]^\top \h p(t) + [\h{B}(\bar \Th_T^*(t) - \bar \Th^*)]^\top p_T(t) + [\h{C} + \h{D} \bar \Th^*]^\top (P_T(t) - P) \sigma \\
& \hspace{0.5in} + [\h{D} (\bar \Th_T^*(t) - \bar \Th^*)]^\top P_T(t) \sigma + [\bar \Th_T^*(t) - \bar \Th^*]^\top \h{r} + (\Pi_T(t) - \Pi) b \\
&:= [\h{A} + \h{B}\bar \Th^*]^\top \h p(t) + \h h(t),
\end{aligned}
\end{equation*}
with the terminal condition $\h p(T) = p_T(T) - p = -p$. The definition of $\h h(\cd)$ leads to
\begin{equation*}
\begin{aligned}
|\h h(t)| &\les \big|[\h{B}(\bar \Th_T^*(t) - \bar \Th^*)]^\top p_T(t) \big| + \big|[\h{C} + \h{D} \bar \Th^*]^\top (P_T(t) - P) \sigma \big| \\
& \hspace{0.5in} + \big|[\h{D} (\bar \Th_T^*(t) - \bar \Th^*)]^\top P_T(t) \sigma \big| + \big|[\bar \Th_T^*(t) - \bar \Th^*]^\top \h{r}\big| + \big|(\Pi_T(t) - \Pi) b \big|.
\end{aligned}
\end{equation*}
From the uniform boundedness of $P_T(\cd)$ and $p_T(\cd)$, and the exponential estimates \eqref{eq:estimates_P}, \eqref{eq:estimates_Pi}, and \eqref{eq:estimates_bar_Theta_T}, we have
$$|\h h(t)| \les K e^{-\l_1(T-t)}, \quad \forall t \in [0, T]$$
for some positive constants $K$ and $\l_1$. By the assumption (H2) and the fact that $\bar \Th^{*}$
is a stabilizer of \eqref{eq:homo_ode} in Lemma \ref{l:infinite_time_homo} (i), the matrix $\h{A} + \h{B}\bar \Th^*$ is stable. Consequently, there exist some constants $K, \l_2 > 0$ with $\l_2 \neq \l_1$ such that
$$\big\|e^{(\h{A} + \h{B}\bar \Th^*)^\top t} \big\| \les K e^{- \l_2 t}, \quad \forall t \ges 0.$$
Thus, there exists some $\l$ satisfying $0 < \l \les \min\{\l_1, \l_2\}$ such that
\begin{equation*}
\begin{aligned}
|\h p(t)| &\les \big|e^{(\h{A} + \h{B} \bar \Th^*)^\top (T-t)} p \big| + \int_t^T \big| e^{(\h{A} + \h{B}\bar \Th^*)^\top (s - t)} \h h(s) \big| ds \\
&\les K e^{- \l_2 (T-t)} + K \int_t^T e^{- \l_2 (s-t)} e^{-\l_1 (T-s)} ds \\
&\les K e^{- \l (T-t)}
\end{aligned}
\end{equation*}
for all $t \in [0, T]$, which concludes \eqref{eq:estimates_p}.

Finally, we prove \eqref{eq:estimates_theta}. By definition, one has
\begin{equation*}
\begin{aligned}
|\th_T^*(t)-\th^*| &= \big| \h{\cR}(P_T(t))^{-1} (\h{B}^\top p_T(t) + \h{D}^\top P_T(t) \sigma + \h{r}) - \h{\cR}(P)^{-1} (\h{B}^\top p + \h{D}^\top P \sigma + \h{r}) \big| \\
& \les \big| \big([\h{R}+\h{D}^\top P_T(t) \h{D}]^{-1}-[\h{R}+\h{D}^\top P\h{D}]^{-1} \big) [\h{B}^\top p_T(t) + \h{D}^\top P_T(t) \sigma + \h{r}] \big|\\
&\hspace{0.5in} + \big|[\h{R}+\h{D}^\top P\h{D}]^{-1} \big(\h{B}^\top (p_T(t) - p) + \h{D}^\top (P_T(t) - P) \sigma \big) \big|\\
& \les Ke^{-\l(T-t)}
\end{aligned}
\end{equation*}
by using the inequalities \eqref{eq:estimates_P} and \eqref{eq:estimates_p}, and the facts that $P_T(\cd)$ and $p_T(\cd)$ are uniformly bounded on $[0, T]$.
\end{proof}

\subsection{Turnpike properties}
\label{s:turnpike_properties}

We now present the first main result of this subsection: a turnpike property that characterizes the connection between the optimal pair of Problem (MFLQ)$^T$ and that of Problem (EC).

\begin{theorem}
\label{t:turnpike_property}
Suppose that {\rm(H1)--(H2)} hold. Let $(X_T(\cd), u_T(\cd))$ denote the optimal pair of Problem (MFLQ)$^T$ with initial state $x \in \mathbb R^n$, and let $(\bar X(\cd), \bar u(\cd))$ be the optimal pair of Problem (EC) with an arbitrary initial state $\bar x \in \mathbb R^n$. Then, there exist positive constants $K$ and $\lambda$, independent of $T$, such that
\begin{equation}
\label{eq:strong_turnpike}
\dbE \big[|X_{T}(t) - \bar X (t)|^2
+ |u_{T} (t) - \bar u (t)|^2
\big] \les K \big(e^{-\lambda t} + e^{-\lambda(T - t)} \big), \quad \forall t \in [0, T].
\end{equation}
\end{theorem}

\begin{proof}
From part (ii) of Lemma \ref{l:finite_time_control}, the optimal trajectory $X_T(\cd)$ for the finite-horizon LQ mean field control problem satisfies the following SDE:
\begin{equation*}
\begin{cases}
\vspace{4pt}
\displaystyle
d X_T(t) = \{(A + B \Th_T^*(t))(X_T(t) - \dbE[X_T(t)]) + (\h{A} + \h{B} \bar \Th_T^*(t)) \dbE[X_T(t)] + \h{B} \theta_T^*(t) + b \} dt \\
\vspace{4pt}
\displaystyle
\hspace{0.7in} + \{(C + D \Th_T^*(t))(X_T(t) - \dbE[X_T(t)]) + (\h{C} + \h{D} \bar \Th_T^*(t)) \dbE[X_T(t)] + \h{D} \theta_T^*(t) + \sigma \} dW(t), \\
X_T(0) = x.
\end{cases}
\end{equation*}
For any $t \in [0, T]$, we let
$\h X(t) = X_T(t) - \bar X(t)$, then it follows that
\begin{equation*}
\begin{aligned}
d \h X(t) &= d (X_T(t) - \bar X(t)) \\
&= \big\{(A + B \Th_T^*(t))(\h X(t) - \dbE[\h X(t)]) + B(\Th_T^*(t) - \Th^*) (\bar X(t) - \dbE[\bar X(t)]) \\
& \hspace{0.5in} + (\h{A} + \h{B} \bar \Th_T^*(t)) \dbE[\h X(t)] + \h{B}(\bar \Th_T^*(t) - \bar \Th^*) \dbE[\bar X(t)] + \h{B}(\theta_T^*(t) - \theta^*) \big\} dt \\
& \hspace{0.2in} + \big\{(C + D \Th_T^*(t))(\h X(t) - \dbE[\h X(t)]) + D(\Th_T^*(t) - \Th^*) (\bar X(t) - \dbE[\bar X(t)]) \\
& \hspace{0.5in} + (\h{C} + \h{D} \bar \Th_T^*(t)) \dbE[\h X(t)] + \h{D}(\bar \Th_T^*(t) - \bar \Th^*) \dbE[\bar X(t)] + \h{D}(\theta_T^*(t) - \theta^*) \big\} dW(t)
\end{aligned} 
\end{equation*}
with the initial state
$$\h X(0) = X_T(0) - \bar X(0) = x - \bar x \in \dbR^n.$$ 
Consequently, the function $t \mapsto \dbE[\h X(t)]$ satisfies the following ODE
\begin{equation*}
\begin{aligned}
\frac{d}{dt} \dbE[\h X(t)] &= (\h{A} + \h{B} \bar \Th^*) \dbE[\h X(t)] + \h{B}(\bar \Th_T^*(t) - \bar \Th^*) \dbE[\h X(t)] \\
& \hspace{0.5in} + \h{B}(\bar \Th_T^*(t) - \bar \Th^*) \dbE[\bar X(t)] + \h{B}(\theta_T^*(t) - \theta^*),
\end{aligned}
\end{equation*}
which leads to
\begin{equation*}
\begin{aligned}
\dbE[\h X(t)] &= e^{(\h{A} + \h{B} \bar \Th^*) t} (x - \bar x) + \int_0^t e^{(\h{A} + \h{B} \bar \Th^*) (t-s)} \big\{\h{B}(\bar \Th_T^*(s) - \bar \Th^*) \dbE[\h X(s)] \\
& \hspace{0.5in} + \h{B}(\bar \Th_T^*(s) - \bar \Th^*) \dbE[\bar X(s)] + \h{B}(\theta_T^*(s) - \theta^*) \big\} ds.
\end{aligned}
\end{equation*}
Since $\h{A} + \h{B} \bar \Th^*$ is stable, by the results in Lemma \ref{l:moment_boundedness} and Lemma \ref{l:estimates}, we obtain
\begin{equation*}
\begin{aligned}
\big|\dbE[\h X(t)] \big| &\les Ke^{-\lambda t} + K \int_0^t e^{- \lambda (t-s)} e^{-\lambda(T-s)} \big(\big|\dbE[\h X(s)] \big| + 1 \big) ds \\
&\les K \big( e^{-\lambda t} + e^{-\lambda (T-t)} \big) + K \int_0^t e^{- \lambda (T + t - 2s)} \big|\dbE[\h X(s)] \big| ds.
\end{aligned}
\end{equation*}
Applying Gr\"onwall's inequality, we establish the uniform boundedness of $\dbE[\h X(\cdot)]$ on $[0, T]$, i.e.,
\begin{equation}
\label{eq:bounded_X_hat}
\big|\dbE[\h X(t)] \big| \les K \big( e^{-\lambda t} + e^{-\lambda (T-t)} \big), \quad \forall t \in [0, T].
\end{equation}

Next, let $P$ be the positive definite solution to the algebraic equation \eqref{eq:ARE_P}. By It\^o's formula,
\begin{equation*}
\begin{aligned}
& \frac{d}{dt} \dbE \big[\lan P \h X(t), \h X(t)\ran \big] \\
= \ & \dbE \big[\lan \{P(A + B \Th^*_T(t)) + (A + B \Th^*_T(t))^\top P + (C + D \Th^*_T(t))^\top P (C + D \Th^*_T(t))\} \h X(t), \h X(t) \ran  \\
& \hspace{0.5in} + 2 \lan \h X(t), H_1(t) \ran + \lan P H_2(t), H_2(t) \ran \big],
\end{aligned}
\end{equation*}
where
\begin{equation*}
\begin{aligned}
H_2(t) &:= (\bar C + \h{D} \bar{\Th}_T^*(t) - D \Th_T^*(t)) \dbE[\h X(t)] + D (\Th_T^*(t) - \Th^*) (\bar X(t) - \dbE[\bar X(t)]) \\
& \hspace{0.5in} + \h{D} (\bar{\Th}_T^*(t) - \bar{\Th}^*) \dbE[\bar X(t)] + \h{D} (\theta_T^*(t) - \theta^*),
\end{aligned}
\end{equation*}
and
\begin{equation*}
\begin{aligned}
H_1(t) &:= P(\bar A + \h{B} \bar{\Th}_T^*(t) - B \Th_T^*(t)) \dbE[\h X(t)] + PB(\Th_T^*(t) - \Th^*) (\bar X(t) - \dbE[\bar X(t)]) \\
& \hspace{0.5in} + P\h{B} (\bar{\Th}_T^*(t) - \bar{\Th}^*) \dbE[\bar X(t)] + P\h{B} (\theta_T^*(t) - \theta^*) + (C + D \Th^*_T(t))^\top P H_2(t).
\end{aligned}
\end{equation*}
Using the inequality \eqref{eq:bounded_X_hat} together with the results from Lemma \ref{l:moment_boundedness} and Lemma \ref{l:estimates} again, we get the following estimate for $H_1(t)$ and $H_2(t)$:
\begin{equation*}
\begin{aligned}
\dbE \big[|H_1(t)|^2 + |H_2(t)|^2 \big] &\les K \big(|\dbE[\h X(t)]|^2 + \|\Th_T^*(t) - \Th^*\|^2 + \|\bar{\Th}_T^*(t) - \bar{\Th}^* \|^2 + |\theta_T^*(t) - \theta^*|^2 \big) \\
& \les K \big( e^{-\lambda t} + e^{-\lambda (T-t)} \big), \quad \forall t \in [0, T]
\end{aligned}
\end{equation*}
for some $K > 0$ and $\lambda > 0$. Moreover, we observe that
\begin{equation*}
\begin{aligned}
& P(A+B\Th^*_T(t)) + (A+B\Th_T^*(t))^\top P
+(C+D\Th_T^*(t))^\top P (C+D\Th_T^*(t)) \\
= \ & P(A+B\Th^*) + (A+B\Th^*)^\top P + (C+D\Th^*)^\top P(C+D\Th^*) + H_3(t),
\end{aligned}
\end{equation*}
and
\begin{equation*}
P(A+B\Th^*) + (A+B\Th^*)^\top P + (C+D\Th^*)^\top P(C+D\Th^*) = -(Q + S^\top \Th^* + (\Th^*)^\top S+(\Th^*)^\top R \Th^*) < 0,
\end{equation*}
where
\begin{equation*}
\begin{aligned}
H_3(t) &= PB(\Th_T^*(t) -\Th^*) +(B(\Th_T^*(t)-\Th^*))^\top P + (D(\Th_T^*(t) - \Th^*))^\top P (C+D\Th^*) \\
& \hspace{0.5in} +(C+D\Th_T^*(t))^\top PD(\Th_T^*(t)-\Th^*).
\end{aligned}
\end{equation*}
Using the estimation results in Lemma \ref{l:estimates}, we deduce that there exist some constants $K > 0$ and $\lambda > 0$ such that
\begin{equation*}
\|H_3(t)\| \les K e^{-\lambda (T - t)}, \quad \forall t \in [0, T].  
\end{equation*}
Let $\beta_1 > 0$ and $\beta_2 > 0$ denote the smallest eigenvalues of $P$ and $Q + S^\top \Th^* + (\Th^*)^\top S+(\Th^*)^\top R \Th^*$ respectively, and let $\h \beta > 0$ be the largest eigenvalue of $P$. Then, it follows that
\begin{equation}
\label{eq:eigen_inequality}
\beta_1 \dbE\big[|\h X(t)|^2\big] \les \dbE \big[\lan P \h X(t), \h X(t)\ran \big] \les \h \beta \dbE\big[|\h X(t)|^2\big], \quad \forall t \in [0, T].
\end{equation}
Combining the above results and applying Young's inequality, we find that
\begin{equation*}
\begin{aligned}
\frac{d}{dt} \dbE \big[\lan P \h X(t), \h X(t)\ran \big] & \les \dbE \Big[\big(- \beta_2 + Ke^{-\lambda(T-t)} \big) |\h X(t)|^2 + 2 |H_1(t)| |\h X(t)| + K|H_2(t)|^2 \Big] \\
& \les \dbE \Big[ \Big(- \frac{\beta_2}{2} + Ke^{-\lambda(T-t)} \Big) |\h X(t)|^2 + \frac{2}{\beta_2} |H_1(t)|^2 + K \big(e^{-\lambda t} + e^{-\lambda(T-t)} \big) \Big] \\
& \les \Big(- \frac{\beta_2}{2} + Ke^{-\lambda(T-t)} \Big) \dbE \big[ |\h X(t)|^2 \big] + K \big(e^{-\lambda t} + e^{-\lambda(T-t)} \big)
\end{aligned}
\end{equation*}
for all $t \in [0, T]$. By the inequality \eqref{eq:eigen_inequality}, we obtain
$$\frac{d}{dt} \dbE \big[\lan P \h X(t), \h X(t)\ran \big] \les \Big(- \frac{\beta_2}{2 \h{\beta}} + \frac{K}{\beta_1}e^{-\lambda(T-t)} \Big) \dbE \big[\lan P \h X(t), \h X(t)\ran \big] + K \big(e^{-\lambda t} + e^{-\lambda(T-t)} \big).$$
Define $g(t) := - \frac{\beta_2}{2 \h{\beta}} + \frac{K}{\beta_1} e^{-\lambda(T-t)}$ for all $t \in [0, T]$. Then, it is clear that, for all $0 \les s \les t < T$, the following estimate holds:
$$\exp \Big\{\int_s^t g(r) dr\Big\} = \exp \Big\{\frac{K}{\beta_1 \lambda} e^{-\lambda T}(e^{\lambda t} - e^{\lambda s}) - \frac{\beta_2}{2 \h{\beta}}(t-s) \Big\} \les \exp \Big\{\frac{K}{\beta_1 \lambda} - \frac{\beta_2}{2 \h{\beta}}(t-s) \Big\}.$$
Thus, we have
$$\int_0^t e^{\int_s^t g(r) dr} K \big(e^{-\lambda s} + e^{-\lambda (T-s)}\big) ds \les K \big(e^{-\lambda t} + e^{-\lambda (T-t)}\big), \quad \forall t \in [0, T],$$
with possible different constants $K, \lambda > 0$.
It follows that
\begin{equation*}
\begin{aligned}
\dbE \big[\lan P \h X(t), \h X(t)\ran \big] & \les \dbE \big[\lan P \h X(0), \h X(0)\ran \big] e^{\int_0^t g(r) dr} + \int_0^t e^{\int_s^t g(r) dr} K \big(e^{-\lambda s} + e^{-\lambda (T-s)}\big) ds \\
& \les K \big(e^{-\lambda t} + e^{-\lambda(T-t)} \big)
\end{aligned}
\end{equation*}
for all $t \in [0, T]$. Hence, since $P$ is positive definite, there exist some positive constants $K$ and $\lambda$, independent of $T$, such that
\begin{equation*}
\dbE \big[|\h X(t)|^2 \big] = \dbE \big[|X_{T}(t) - \bar X (t)|^2 \big] \les K \big(e^{-\lambda t} + e^{-\lambda(T-t)} \big), \quad \forall t \in [0, T].
\end{equation*}
Next, from the above estimation and the moment boundedness of $\bar X(\cd)$ in Lemma \ref{l:moment_boundedness}, 
$$\dbE \big[|X_T(t)|^2 \big] \les K, \quad \forall t \in [0, T].$$
Note that, for all $t \in [0, T]$, the difference $u_T(t) - \bar{u}(t)$ satisfies the following identity:
\begin{equation*}
\begin{aligned}
u_T(t) - \bar u(t) &= \Th_T^*(t) \h X(t) + (\Th_T^*(t) - \Th^*) \bar X(t) + (\bar{\Th}_T^*(t) - \Th_T^*(t)) \dbE[\h X(t)] \\
& \hspace{0.5in} + (\bar{\Th}_T^*(t) - \bar{\Th}^* + \Th^* - \Th_T^*(t)) \dbE[\bar X(t)] + \theta_T^*(t) - \theta^*.
\end{aligned}
\end{equation*}
By applying Lemma \ref{l:moment_boundedness} and Lemma \ref{l:estimates}, we conclude that the turnpike property for the optimal control holds. Specifically, there exist $K, \lambda > 0$, independent of $t$ and $T$, such that
\begin{equation*}
\dbE \big[|u_T(t) - \bar u(t)|^2 \big] \les K \big(e^{-\lambda t} + e^{-\lambda (T-t)} \big), \quad \forall t \in [0, T].
\end{equation*}
This completes the proof.
\end{proof}

\begin{remark}
\label{r:comparison}
\textit{It is worth noting that our result in Theorem \ref{t:turnpike_property} allows an arbitrary initial state for the optimal path $\bar{X}(\cd)$ of Problem (EC), i.e., the turnpike property remains valid for any choice of the initial state of $\bar{X}(\cd)$. In particular, if $\bar{X}(0) = x^*$, where $x^*$ solves the static optimization problem \eqref{eq:static_optimization_problem}, we can recover the result in Theorem 3.2 of \cite{Sun-Yong-2024}.}

\textit{Let $\bar{q} = 0$ and $\bar{r} = 0$. The pair of processes $(\BX^*(\cd), \Bu^*(\cd))$ used to establish the turnpike property \eqref{eq:turnpike_stochastic} in \cite{Sun-Yong-2024} are defined as follows:
$$\BX^*(t) = X^*(t) + x^*, \quad \Bu^*(t) = \Th^* X^*(t) + u^*, \quad \forall t \ges 0,$$
where $(x^*, u^*)$ is the solution to the static optimization problem \eqref{eq:static_optimization_problem}, $\Theta^*$ is defined in \eqref{eq:theta_star}, and $X^*(\cd)$ is the solution to the following SDE
\begin{equation*}
\begin{cases}
d X^*(t) = (A + B \Th^*) X^*(t) dt + [(C+ D \Th^*) X^*(t) + \sigma^*] dW(t), \quad t \ges 0, \\
X^*(0) = 0,
\end{cases}
\end{equation*}
with $\sigma^* := \h{C} x^* + \h{D}u^* + \sigma$.
}

\textit{Firstly, we claim that $u^* = \bar{\Theta}^* x^* + \theta^*$, where $\bar{\Theta}^*$ and $\theta^*$ are given in \eqref{eq:theta_star}. By the feedback form of the optimal control of Problem (MFLQ)$^T$ in \eqref{eq:optimal_control_finite}, we observe that
\begin{equation*}
\begin{aligned}
u^* - \big(\bar{\Theta}^* x^* + \theta^* \big) &= \Bu^*(t) - \Th^* X^*(t) - \bar{\Theta}^* x^* - \theta^* \\
&= \Bu^*(t) - u_T(t) + \Th_T^*(t) X_T(t) - \Th^* \BX^*(t) + \Th^* \dbE[\BX^*(t)] - \Th_T^*(t) \dbE[X_T(t)] \\
& \hspace{0.5in} + \bar{\Theta}_T^*(t) \dbE[X_T(t)] - \bar{\Theta}^* \dbE[\BX^*(t)] + \th_T^*(t) - \theta^*
\end{aligned}
\end{equation*}
as $\dbE[\BX^*(t)] = x^*$. Then, using the results in \eqref{eq:turnpike_stochastic} and Lemma \ref{l:estimates}, we derive the following estimate:
$$\big|u^* - \big(\bar{\Theta}^* x^* + \theta^* \big) \big| \les K \big(e^{-\lambda t} + e^{-\lambda (T-t)} \big), \quad \forall t \in [0, T].$$
Thus, setting $t = T/2$ and letting $T \to \infty$, we obtain $u^* = \bar{\Theta}^* x^* + \theta^*$.} 

\textit{Next, from the constraint $\h{A} x^* + \h{B}u^* + b = 0$ of the static optimization problem \eqref{eq:static_optimization_problem}, it is clear that
$$(\h{A} + \h{B} \bar{\Th}^*) x^* + (\h{B}\theta^* + b) = 0,$$
which yields $x^* = - (\h{A} + \h{B} \bar{\Th}^*)^{-1} (\h{B}\theta^* + b)$ as the matrix $\h{A} + \h{B} \bar{\Th}^*$ is stable.}

\textit{Finally, we show that $\BX^*(\cd)$ satisfies the same SDE as $\bar{X}(\cd)$ in \eqref{eq:optimal_path_ergodic} if we set $\bar{X}(0) = x^*$. By the definition of $\BX^*(\cd)$, it is straightforward that
\begin{equation*}
\begin{aligned}
d \BX^*(t) &= (A + B \Th^*) (\BX^*(t) - x^*) dt + [(C+ D \Th^*) (\BX^*(t) - x^*) + \sigma^*] dW(t) \\
&= \big\{(A + B \Th^*) (\BX^*(t) - \dbE[\BX^*(t)]) + (\h{A} + \h{B} \bar \Th^*) \dbE[\BX^*(t)] + \h{B} \theta^* + b \big\} dt \\
& \hspace{0.2in} + \big\{(C + D \Th^*) (\BX^*(t) - \dbE[\BX^*(t)]) + (\h{C} + \h{D} \bar \Th^*) \dbE[\BX^*(t)] + \h{D} \theta^* + \sigma \big\} dW(t),
\end{aligned}
\end{equation*}
which concludes the desired result.
}
\end{remark}

In addition to the turnpike property described by \eqref{eq:strong_turnpike}, we next establish a related turnpike behavior concerning the cost functional:
$$c_0 = \lim_{T\to\i}\frac{1}{T} J_T(x; \bar u(\cd)) = \bar c := \lim_{T\to\i} \frac{1}{T} J_T(x; u_T(\cd)).$$
This result demonstrates that achieving near-optimality in terms of the long-term average cost does not necessitate solving for the optimal control $u_T(\cd)$ for each terminal time $T$. Instead, it is sufficient to compute the optimal control $\bar u(\cd)$ from Problem (EC).

\begin{lemma}
\label{l:convergence_of_value_function}
Assume that {\rm(H1)--(H2)} hold. Let $J_{T}(x; \bar u(\cd))$ denote the cost functional evaluated along the optimal control
$\bar u(\cd)$ of Problem (EC) over the finite horizon $[0, T]$, i.e.,
$$J_T(x; \bar u(\cd)) := \dbE \Big[\int_0^T f(\bar X(t), \dbE[\bar X(t)], \bar u(t), \dbE[\bar{u}(t)]) d t\Big].$$
For all $x \in \mathbb R^n$, the following estimate holds:
$$0 \les J_T(x; \bar u(\cd)) - V_T(x) = O(1).$$
Furthermore, the constant $c_0$ of Problem (C) is identical to the ergodic cost $\bar c$ defined via \eqref{eq:ergodic_cost}, i.e.,
$$c_0 = \bar c = \lim_{T\to \infty} \frac{1}{T} V_{T}(x), \quad x \in \mathbb R^n.$$
\end{lemma}
\begin{proof}
By definition, for all $x \in \mathbb R^n$, we have
\begin{equation*}
\begin{aligned}
& |J_T(x; \bar u(\cd)) -
J_T(x; u_T(\cd))| \\
\les \ & \dbE\Big[ \int_0^T |f(\bar X(t), \dbE[\bar X(t)], \bar u(t), \dbE[\bar{u}(t)]) - f(X_T(t), \dbE[X_T(t)], u_T(t), \dbE[u_T(t)])| dt \Big].
\end{aligned}
\end{equation*}
Using a similar argument as in Theorem 6.5 of \cite{Jian-Jin-Song-Yong-2024}, and applying H\"older’s inequality along with the results in Lemma \ref{l:moment_boundedness} and Theorem \ref{t:turnpike_property}, for all $x \in \mathbb R^n$, we obtain the desired estimate that
$$|J_T(x; \bar u(\cd)) - J_T(x; u_T(\cd))| \les K,$$
where $K$ is independent of $T$. Moreover, by the definition of the value function $V_{T}(x) := J_{T} \left(x; u_T(\cd) \right)$ of Problem (MFLQ)$^T$, it is clear that for all $x \in \mathbb R^n$,
$0 \les J_T(x; \bar u(\cd)) - J_T(x; u_T(\cd))$.
Therefore, for all $x \in \mathbb R^n$, the following holds:
$$0 \les J_T(x;\bar u(\cd)) - V_T(x) = O(1),$$
which yields
\begin{equation*}
c_0 = \lim_{T\to \infty} \frac{1}{T} J_T(x; \bar u(\cd)) = \lim_{T\to \infty} \frac{1}{T} V_{T}(x) = \bar c
\end{equation*}
by using \eqref{eq:ergodic_constant} in Theorem \ref{t:ergodic_control_problem} (ii).
\end{proof}


\section{Numerical experiments}
\label{s:numerical_experiments}

To illustrate the convergence of the solutions to the systems \eqref{eq:Riccati_ODE} and \eqref{eq:algebraic_Riccati} established in Section \ref{s:Estimates}, as well as the turnpike property for the optimal pairs and the convergence of the time-averaged value function proved in Theorem \ref{t:turnpike_property} and Lemma \ref{l:convergence_of_value_function}, respectively, we present several numerical experiments in this section.

For simplicity, we consider a one-dimensional ($n = m = 1$) linear-quadratic mean field control problem with the following parameters:
\begin{equation*}
\begin{aligned}
& A = 0.4,\,\, \bar{A} = 0.2,\,\, B = 1,\,\, \bar{B} = 0.2,\,\, C = 0.25,\,\, \bar{C} = 0.1,\,\, D = 0.1, \,\, \bar{D} = 0.05, \\
& b = 0.4,\,\, \sigma = 0.3, \,\, X_T(0) = x = 1,\,\, \bar{X}(0) = \bar{x} = 0.5, \\
& Q = 1,\,\, \bar{Q} = 0.6,\,\, S = 0.1,\,\, \bar{S} = 0.05,\,\, R = 1,\,\, \bar{R} = 0.3, \\
& q = - 0.5, \,\, \bar{q} = 0.1, \,\, r = 0.1, \,\, \bar{r} = 0.05.
\end{aligned}
\end{equation*}
Note that we choose different initial states for the optimal trajectories of Problem (MFLQ)$^T$ and Problem (EC). 
It is clear that these parameters satisfy the assumptions (H1)--(H2). Specifically, we note that $R > 0$, $\h{R} = 1.3 > 0$, and
\begin{equation*}
Q - S^\top R^{-1} S = 0.99 > 0, \quad \h{Q} - \h{S}^\top \h{R}^{-1} \h{S} = 1.5827 > 0,
\end{equation*}
thus the assumption (H1) is valid. Solving the algebraic system \eqref{eq:algebraic_Riccati} numerically gives
$$P = 1.3463, \,\, \Pi = 1.6871, \,\, p = 0.0427, \,\, c_0 = 0.1038.$$
Consequently, we obtain that
$$\Theta^* = -1.4603, \, \bar{\Theta}^* = - 1.6878, \, \theta^* = - 0.1968,$$
which gives that $\h{A} + \h{B} \bar{\Theta}^* < 0$ and $2(A + B\Theta^*) + (C + D\Theta^*)^2 < 0$. Hence, the solution to
\begin{equation*}
\begin{cases}
\vspace{4pt}
\displaystyle
d X(t) = (A + B\Th^*) X(t) dt + (C+ D \Th^*)X(t) dW(t), \quad t \ges 0 \\
X(0) = x
\end{cases}
\end{equation*}
satisfies $\dbE[\int_0^\infty |X(t)|^2 dt] < \infty$. Therefore, the assumption (H2) is also verified. Next, we solve the system of differential equations \eqref{eq:Riccati_ODE} backward in time with zero terminal conditions, and then obtain the time-varying coefficients $\Theta_T^*(\cd)$, $\bar{\Theta}_T^*(\cd)$, and $\theta_T^*(\cd)$.

For $T = 40$, Figure \ref{fig:convergence_of_ODEs} illustrates the convergence of the solution to the finite-horizon system
\eqref{eq:Riccati_ODE} towards the solution of the ergodic system \eqref{eq:algebraic_Riccati}. This is fully consistent with the exponential
turnpike-type estimates established in \eqref{eq:estimates_P}, \eqref{eq:estimates_Pi}, and \eqref{eq:estimates_p}.

\begin{figure}[htb]
\begin{minipage}{\linewidth}
\centering
\subcaptionbox{}	
{\includegraphics[width=0.45\linewidth, height=0.26\textheight]{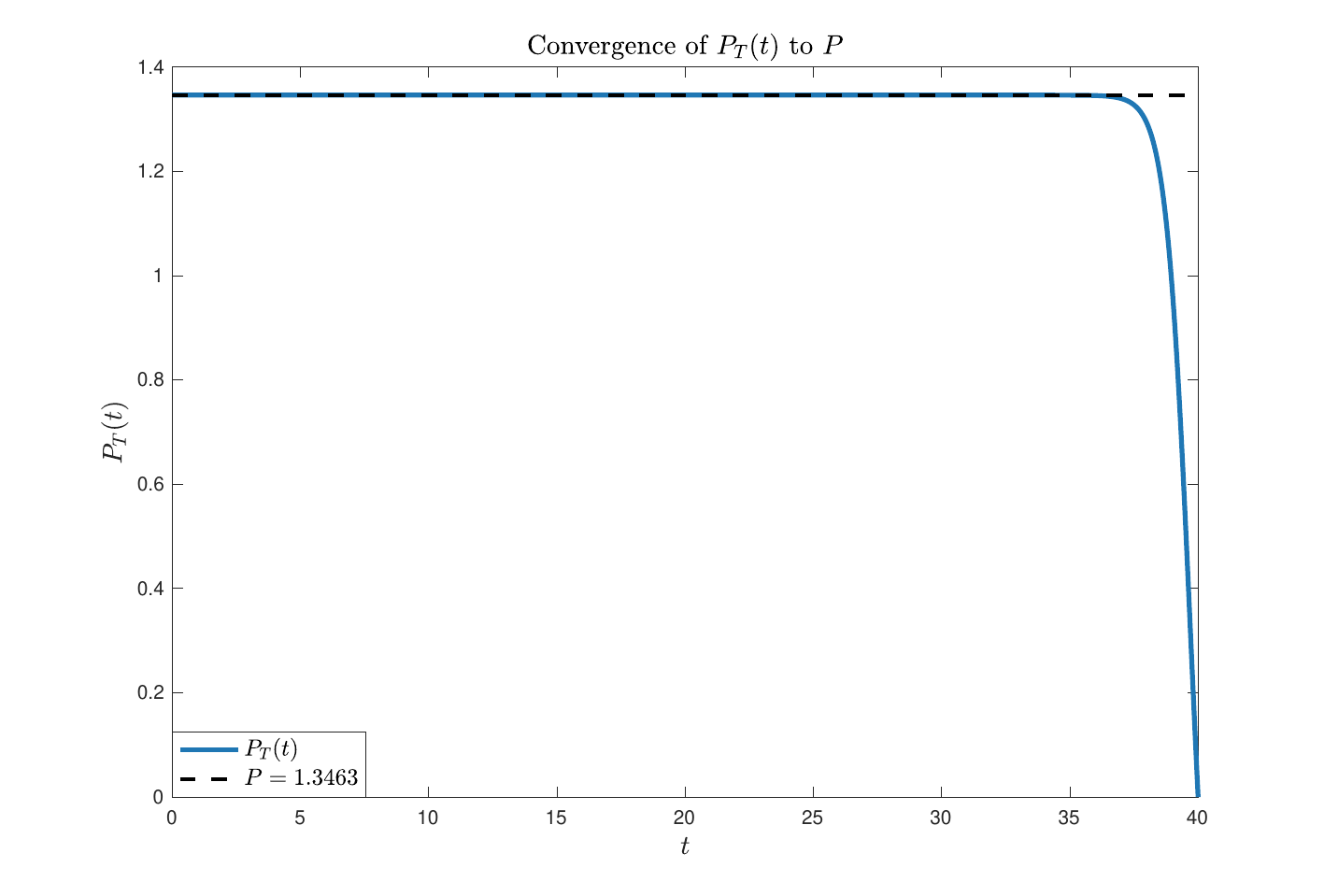}}\quad
\subcaptionbox{}
{\includegraphics[width=0.45\linewidth, height=0.26\textheight]{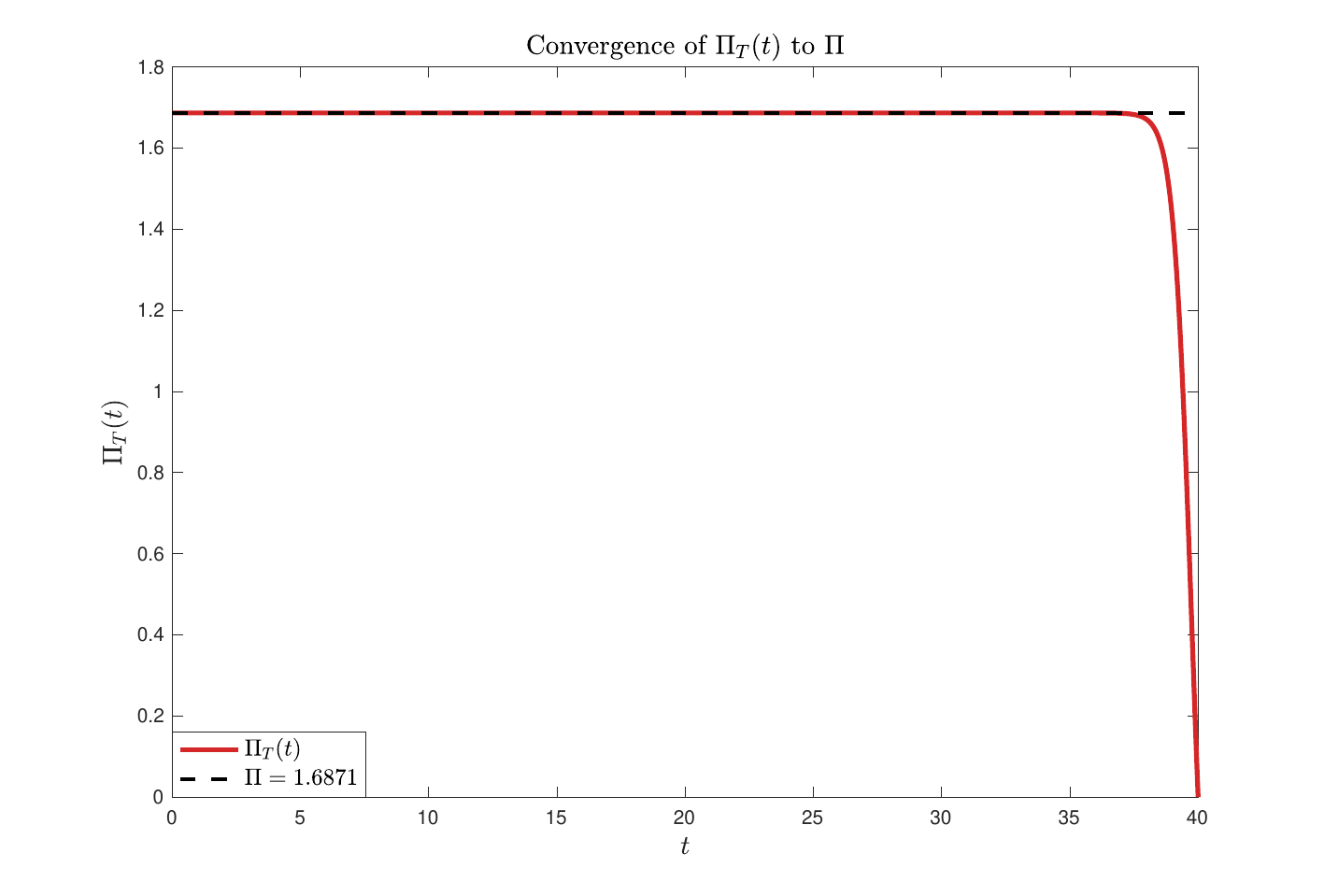}} \\
\subcaptionbox{}	
{\includegraphics[width=0.45\linewidth, height=0.26\textheight]{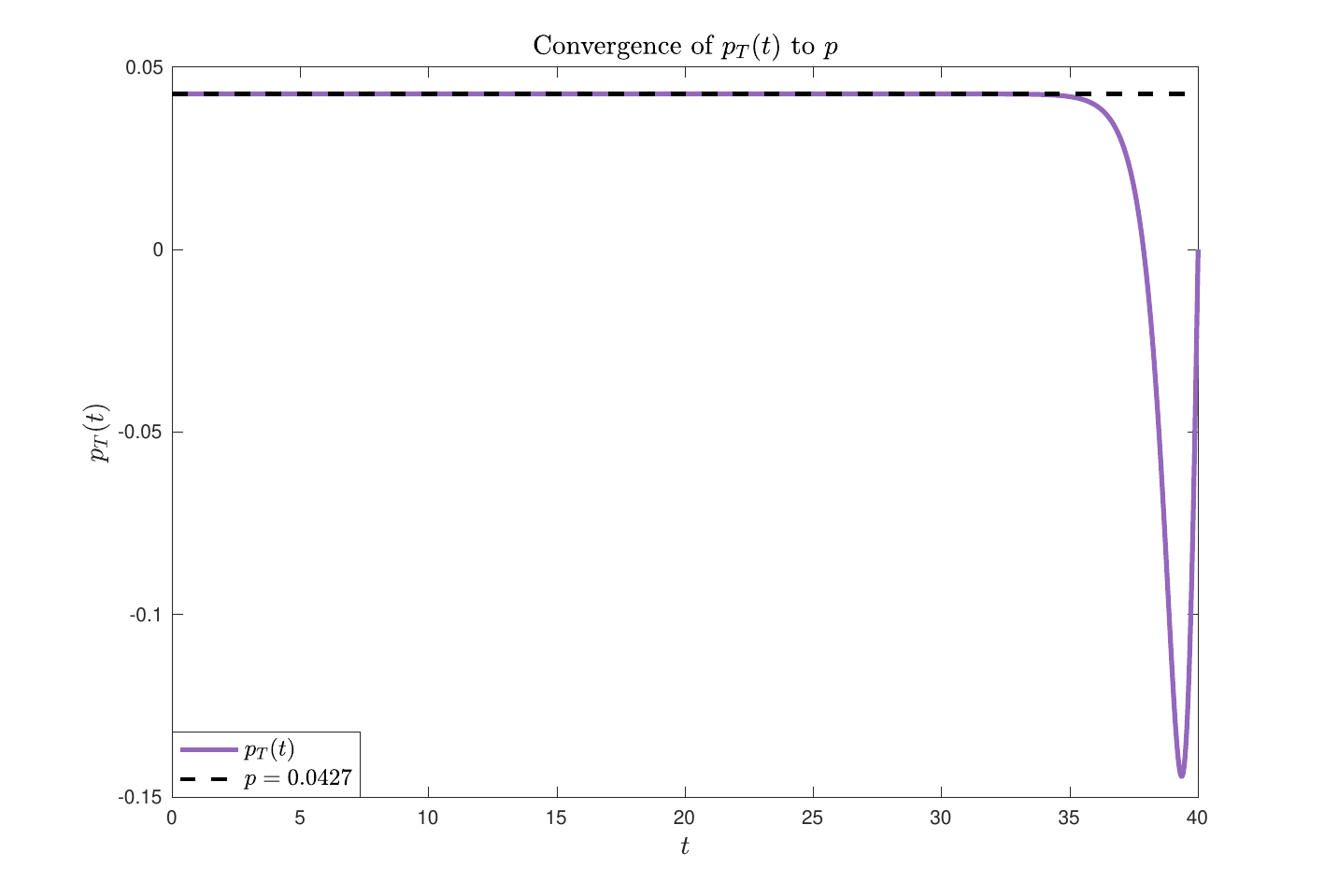}}
\caption{The convergence of solutions to the systems \eqref{eq:Riccati_ODE} and \eqref{eq:algebraic_Riccati}.}
\label{fig:convergence_of_ODEs}
\end{minipage}
\end{figure}

Next, for different time horizons $T \in \{10, 20, 40, 80\}$, Figure \ref{fig:turnpike} illustrates the turnpike property \eqref{eq:strong_turnpike}, established in Theorem \ref{t:turnpike_property}, for the optimal pairs $(X_T(\cd), u_T(\cd))$ and $(\bar{X}(\cd), \bar{u}(\cd))$. In particular, the dashed curve represents a turnpike-type bound of the form $K(e^{-\lambda t} + e^{-\lambda(T-t)})$, where $K$ and $\lambda$ are both chosen independently of $T$, while the solid curve gives the Monte Carlo approximation of $\mathbb{E}[|X_{T}(t) - \bar X (t)|^2
+ |u_{T} (t) - \bar u (t)|^2]$.

\begin{figure}[htb]
\begin{minipage}{\linewidth}
\centering
\subcaptionbox{}	
{\includegraphics[width=0.45\linewidth, height=0.26\textheight]{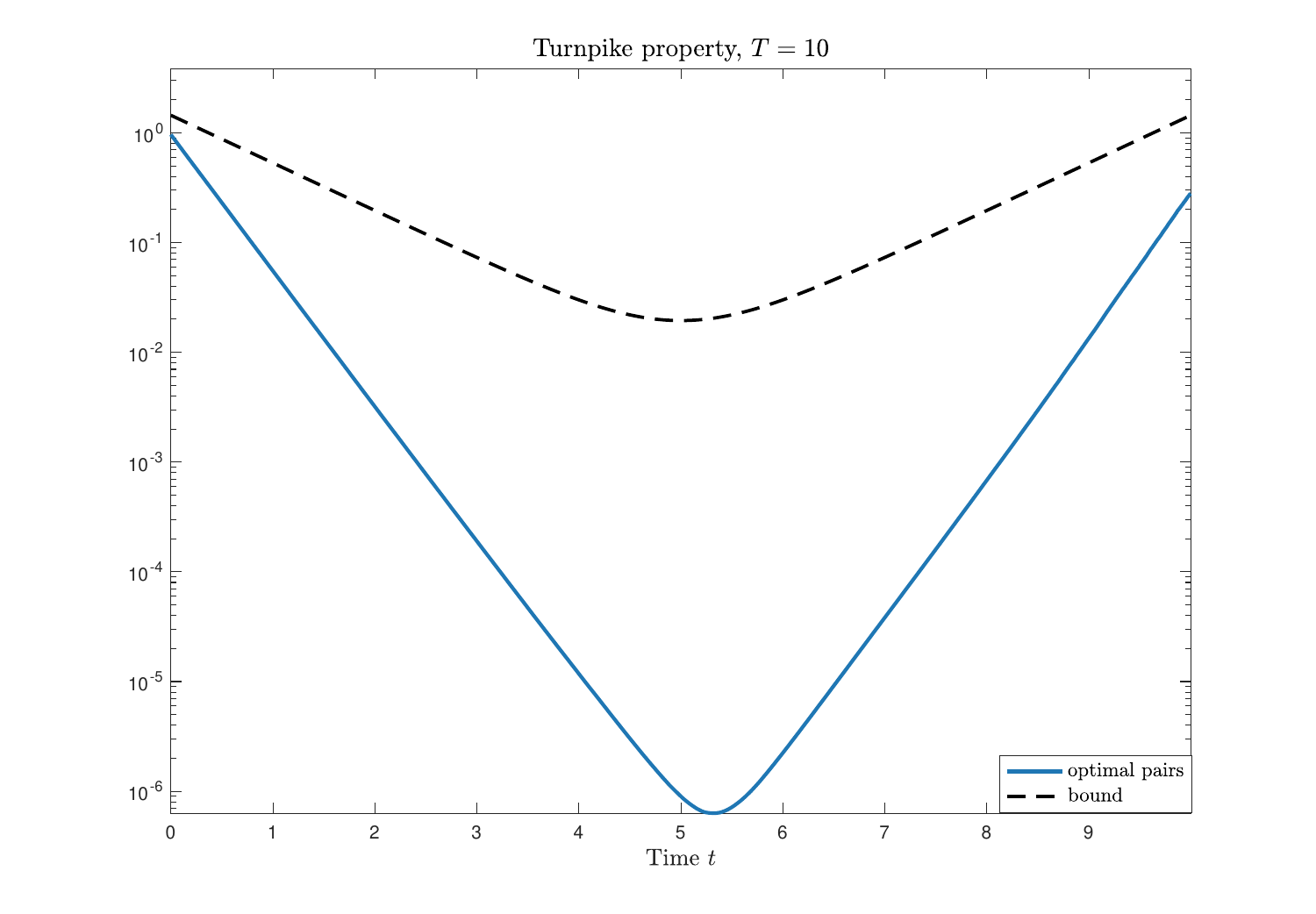}}\quad
\subcaptionbox{}
{\includegraphics[width=0.45\linewidth, height=0.26\textheight]{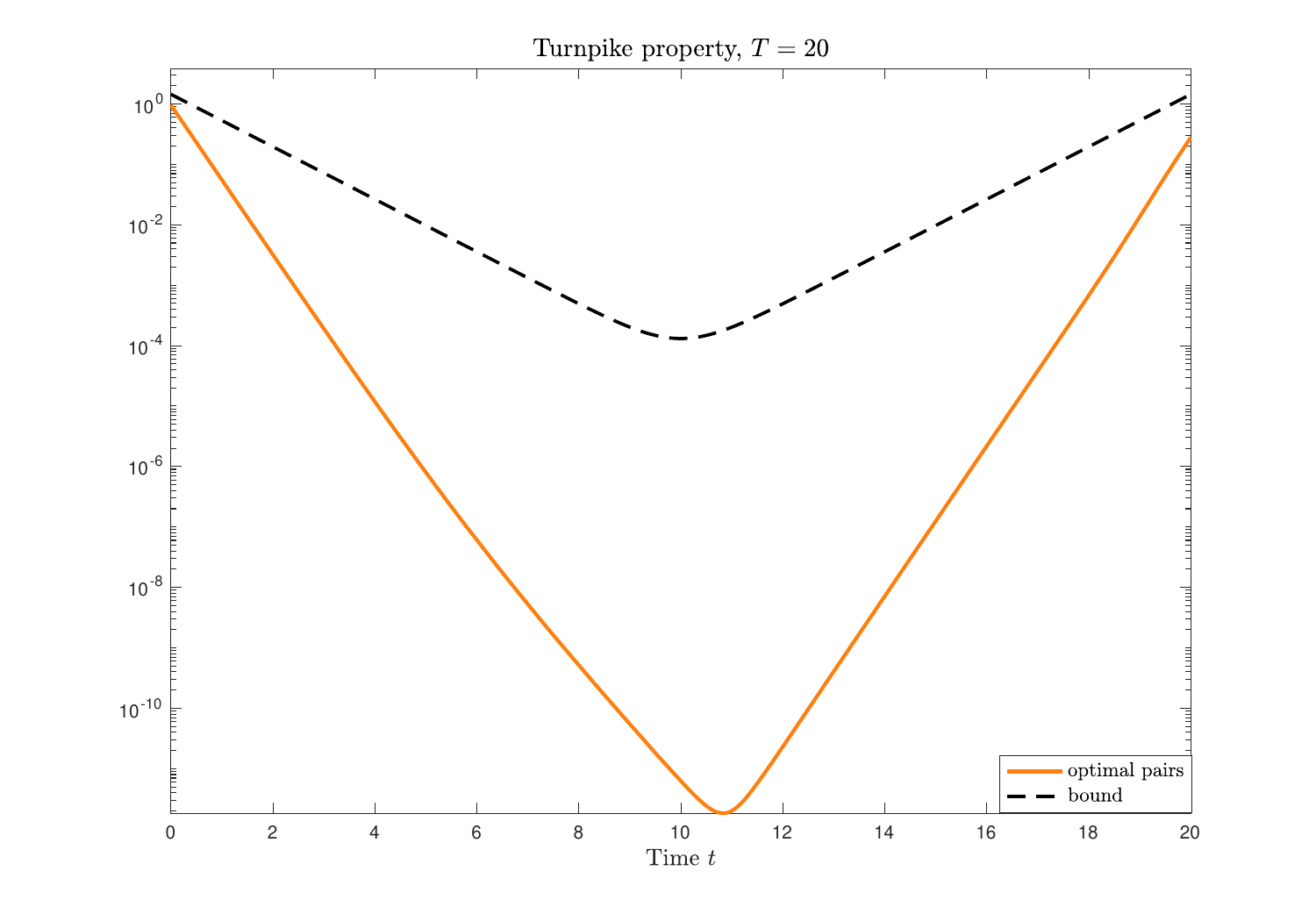}} \\
\subcaptionbox{}	
{\includegraphics[width=0.45\linewidth, height=0.26\textheight]{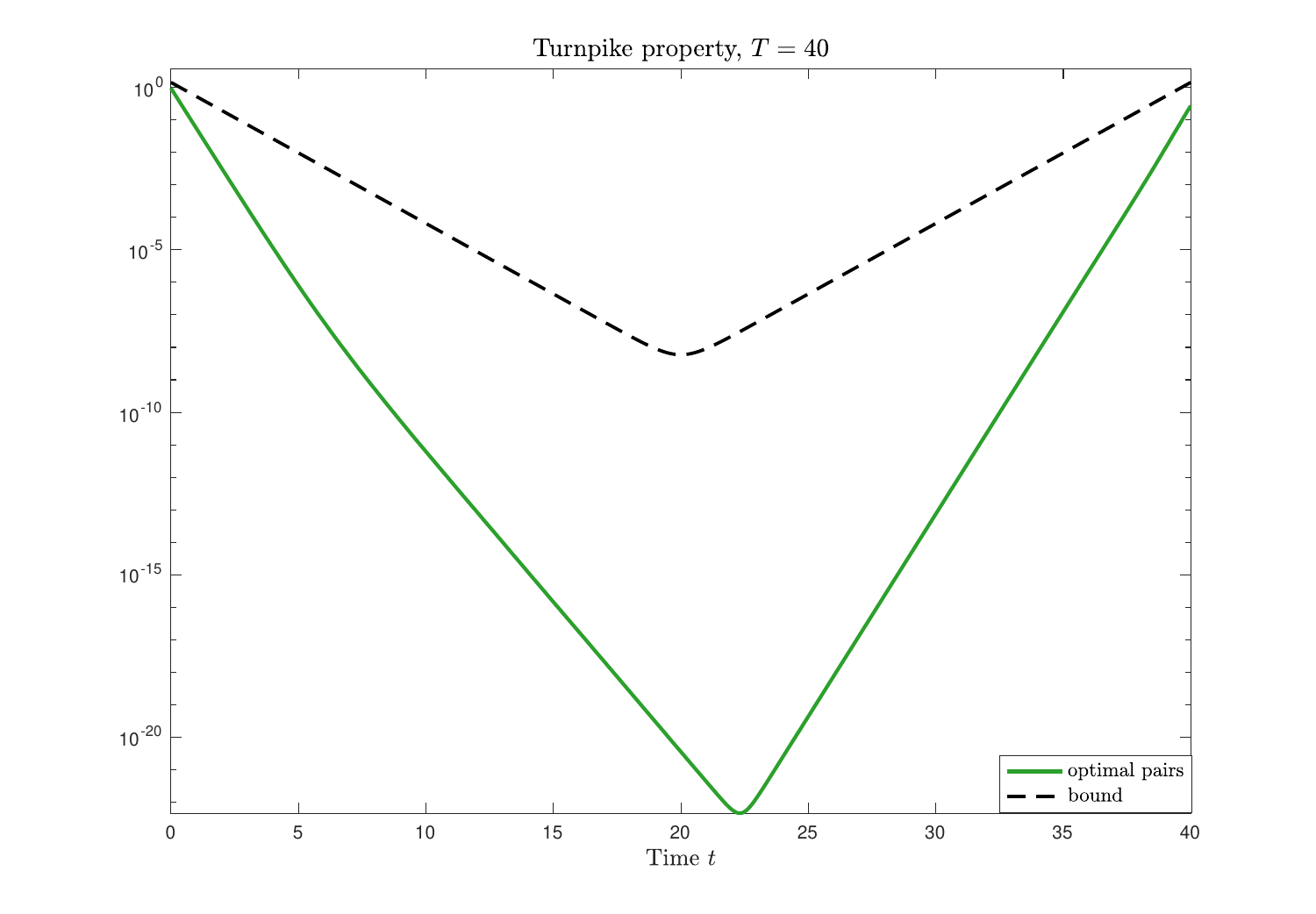}}\quad
\subcaptionbox{}
{\includegraphics[width=0.45\linewidth, height=0.26\textheight]{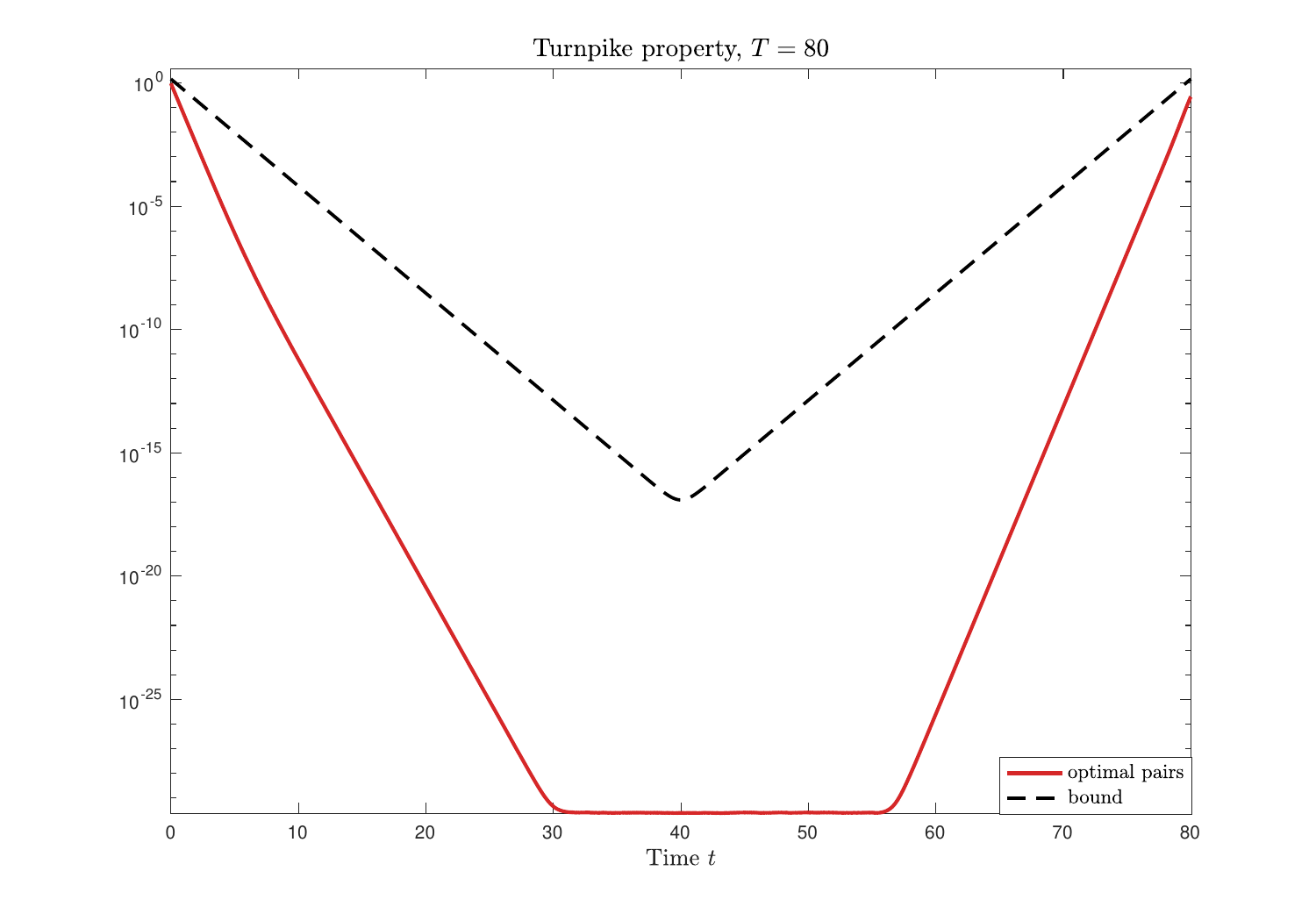}}
\caption{Turnpike property in \eqref{eq:strong_turnpike}.}
\label{fig:turnpike}
\end{minipage}
\end{figure}

In addition, we perform Monte Carlo simulations with $5000$ sample paths to compare the means of the optimal state trajectories in both the finite-horizon and the ergodic control problems. Specifically, we approximate $\mathbb E[X_T(t)]$ and $\mathbb E[\bar{X}(t)]$ for each $t \in [0, T]$. The simulation results are presented in Figure \ref{fig:Monte_Carlo_simulation_path}, which also illustrates the turnpike property at the level of the mean state trajectories. More precisely, it suggests that
$$\big|\mathbb{E}[X_T(t)] - \mathbb E[\bar{X}(t)] \big|^2 \leq K \big(e^{-\lambda t} + e^{-\lambda(T-t)} \big), \quad \forall t \in [0, T],$$
for some constants $K$ and $\lambda$, which are independent of $t$ and $T$. This estimate follows directly from the turnpike property \eqref{eq:strong_turnpike}. The shaded bands in Figure \ref{fig:Monte_Carlo_simulation_path} represent the 10th--90th percentile range of the simulated state processes.

\begin{figure}[htb]
\begin{minipage}{\linewidth}
\centering
\subcaptionbox{}	
{\includegraphics[width=0.40\linewidth, height=0.28\textheight]{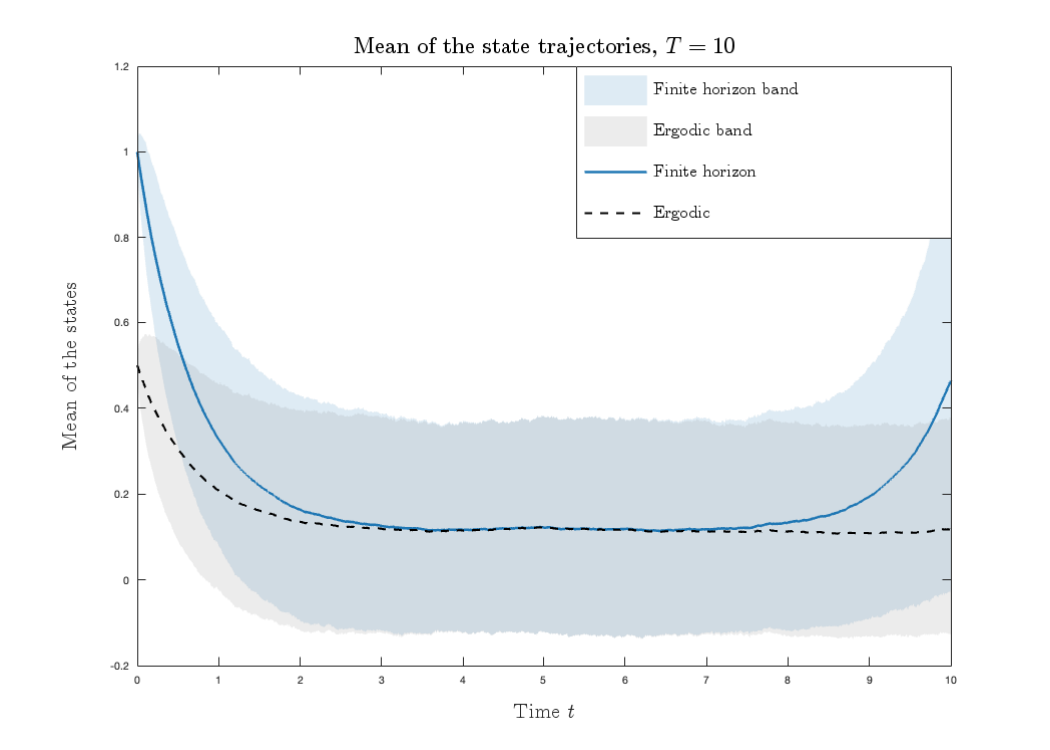}}\quad
\subcaptionbox{}
{\includegraphics[width=0.40\linewidth, height=0.28\textheight]{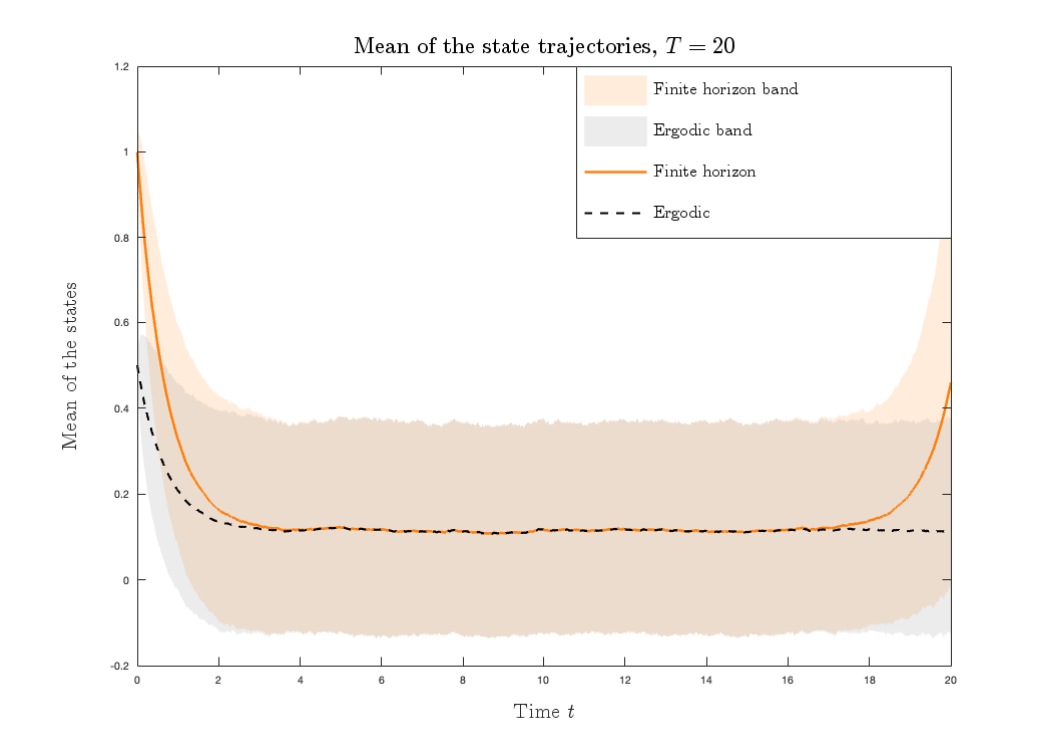}} \\
\subcaptionbox{}	
{\includegraphics[width=0.40\linewidth, height=0.28\textheight]{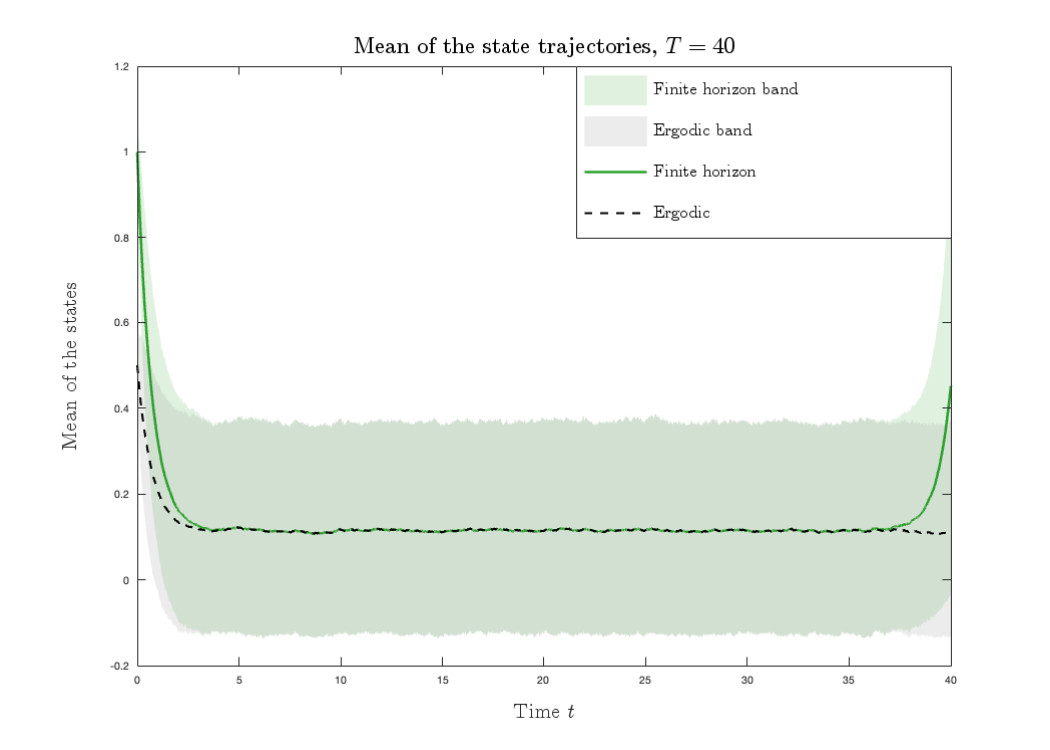}}\quad
\subcaptionbox{}
{\includegraphics[width=0.40\linewidth, height=0.28\textheight]{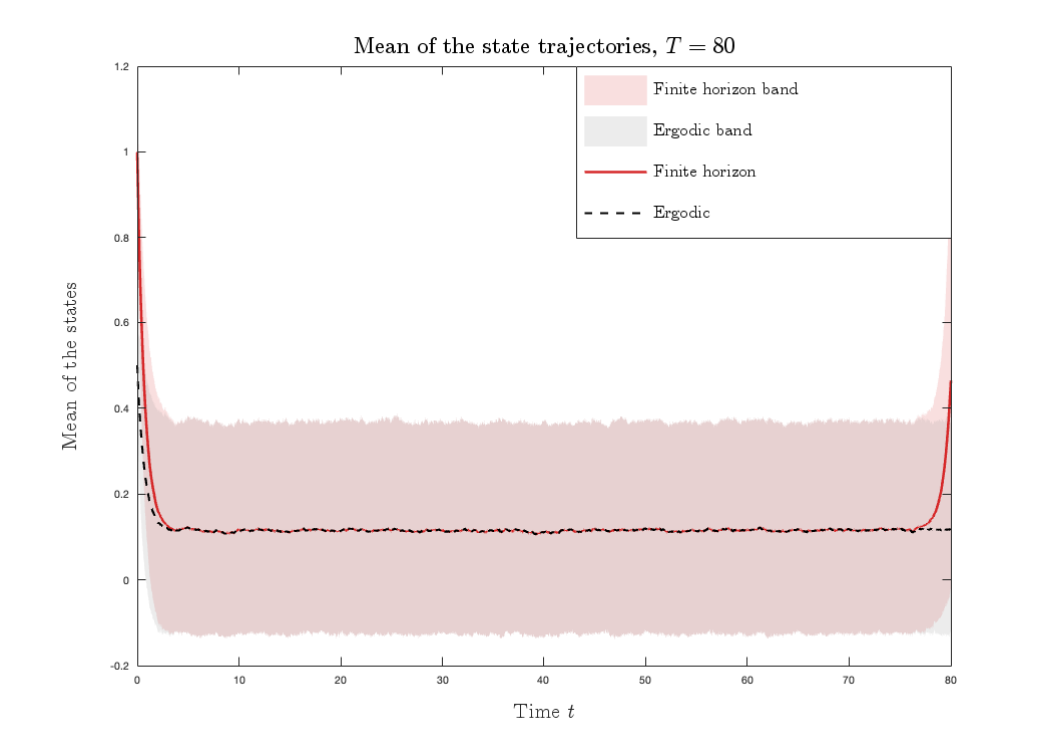}}
\caption{Simulation of optimal states and turnpike property for the mean state trajectories.}
\label{fig:Monte_Carlo_simulation_path}
\end{minipage}
\end{figure}

Finally, we present the simulation for the convergence of the time-averaged value functions with $T \in \{4, 6, 8, 10, 20, 30, 40, 60, 80, 100\}$. The result is displayed in Figure \ref{fig:convergence_of_value}. It shows that, as the time horizon $T$ increases, both $\frac{1}{T}V_T(x)$ and $\frac{1}{T}J_T(x;\bar u(\cdot))$ approach the ergodic constant $c_0$. This numerical trend is consistent with the convergence of the time-averaged value function established in Lemma \ref{l:convergence_of_value_function}. 

\begin{figure}[htb]
\begin{center}
\includegraphics[width=.6\linewidth]{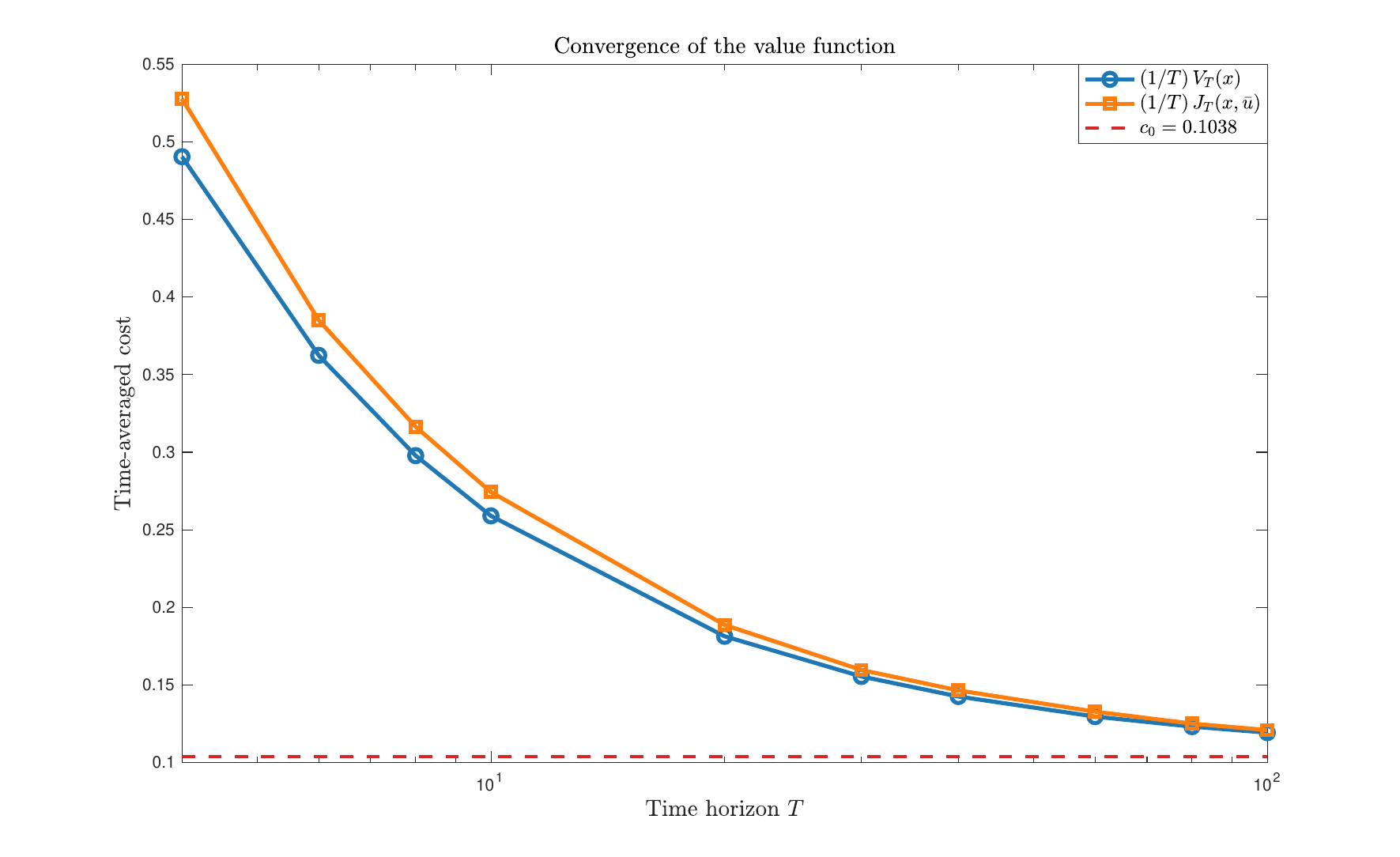} 
\caption{Convergence of time-averaged value function in Lemma \ref{l:convergence_of_value_function}.}
\label{fig:convergence_of_value}
\end{center} 
\end{figure}

\vspace{4pt}

{\bf Acknowledgment.} The authors are grateful to the anonymous Associate Editor and the reviewers for their insightful comments, which significantly improved the quality of this paper.

\bibliographystyle{plain}
\bibliography{reference}

\end{document}